\documentclass[11pt]{article}
\usepackage[utf8]{inputenc}

\usepackage{graphicx}
\usepackage{xcolor}
\usepackage{amsmath,amsthm,amssymb}
\usepackage[margin=1in]{geometry}
\usepackage[]{hyperref}
\usepackage{comment}
\usepackage[normalem]{ulem}
\usepackage{cite}
\usepackage{subfigure}

\newcommand{\R}{\mathbb{R}}

\newcommand{\E}{\mathbb{E}}

\def\a{{\bf a}}
\def\b{{\bf b}}
\def\c{{\bf c}}
\def\x{{\bf x}}
\def\y{{\bf y}}
\def\z{{\bf z}}

\def\e{{\bf e}}

\def\u{{\bf u}}
\def\v{{\bf v}}
\def\w{{\bf w}}

\def\q{{\bf q}}

\def\0{{\bf 0}}
\def\E{{\mathbb E}}
\def\S{{\mathbb S}}
\def\R{{\mathcal R}}
\def\V{{\mathcal V}}

\def\TT{{\tt T}}
\usepackage[title]{appendix}

\newcommand{\be}{\begin{equation}}
\newcommand{\ee}{\end{equation}}
\newcommand{\bea}{\begin{eqnarray}}
\newcommand{\eea}{\end{eqnarray}}
\newcommand{\bes}{\begin{equation*}}
\newcommand{\ees}{\end{equation*}}
\newcommand{\beas}{\begin{eqnarray*}}
\newcommand{\eeas}{\end{eqnarray*}}
\newcommand{\csp}[1]{{\color{black}  #1}}
\newcommand{\cs}[1]{{#1}}

\makeatletter
\newtheorem*{rep@theorem}{\rep@title}
\newcommand{\newreptheorem}[2]{%
\newenvironment{rep#1}[1]{%
 \def\rep@title{#2 \ref{##1} (restated)}%
 \begin{rep@theorem}}%
 {\end{rep@theorem}}}
\makeatother

\allowdisplaybreaks

\newtheorem{thm}{Theorem}
\newtheorem*{thm*}{Theorem}
\newtheorem{cor}[thm]{Corollary}

\newtheorem{lem}[thm]{Lemma}
\newtheorem*{lem*}{Lemma}

\newtheorem{prop}[thm]{Proposition}

\newtheorem{defn}[thm]{Definition}
\newtheorem{rem}[thm]{Remark}

\newreptheorem{thm}{Theorem}
\newreptheorem{lem}{Lemma}
\newreptheorem{cor}{Corollary}

\allowdisplaybreaks
\usepackage{algpseudocode}
\usepackage{algorithm}
\usepackage[style=base]{caption}


\usepackage{pgfplots}
\usepackage{authblk}

\title{ A deterministic Kaczmarz algorithm for solving linear systems }

\author{Changpeng Shao\thanks{changpeng.shao@bristol.ac.uk}}

\affil{School of Mathematics, University of Bristol, UK}

\date{\today}

\begin{document}

\maketitle

\begin{abstract}

{
We propose a new deterministic Kaczmarz algorithm for solving consistent linear systems $A\x=\b$. Basically, the algorithm replaces orthogonal projections with reflections in the original scheme of Stefan Kaczmarz. 
Building on this, we give a geometric description of solutions of linear systems. Suppose $A$ is $m\times n$, we show that the algorithm generates a series of points distributed with patterns on an $(n-1)$-sphere centered on a solution. These points lie evenly on $2m$ lower-dimensional spheres $\{\S_{k0},\S_{k1}\}_{k=1}^m$, with the property that for any $k$, the midpoint of the centers of $\S_{k0},\S_{k1}$ is exactly a solution of $A\x=\b$. 
With this discovery, we prove that taking the average of $O(\eta(A)\log(1/\varepsilon))$ points on any $\S_{k0}\cup\S_{k1}$ effectively approximates a solution up to relative error $\varepsilon$, where $\eta(A)$ characterizes the eigengap of the orthogonal matrix produced by the product of $m$ reflections generated by the rows of $A$.
We also analyze the connection between $\eta(A)$ and $\kappa(A)$, the condition number of $A$. In the worst case $\eta(A)=O(\kappa^2(A)\log m)$, while for random matrices $\eta(A)=O(\kappa(A))$ on average.
Finally, we prove that the algorithm indeed solves the linear system $A^{\TT}W^{-1}A \x = A^{\TT}W^{-1} \b$, where $W$ is the lower-triangular matrix such that $W+W^{\TT}=2AA^{\TT}$. The connection between this linear system and the original one is studied. 
The numerical tests indicate that this new Kaczmarz algorithm has comparable performance to randomized (block) Kaczmarz algorithms.

}



\vspace{.2cm}


\vspace{.2cm}


\end{abstract}

\section{Introduction}

Solving systems of linear equations is a fundamental problem in science and engineering. In practice, the size of linear equations is usually so large that iterative methods are more favourable. Among all the iterative methods, the {\em Kaczmarz method} (also known as {\em alternating projection} or {\em the alternating method of von Neumann}) is popular due to its simplicity and efficiency. 

Assume that $A$ is an $m\times n$ real-valued matrix, $\b$ is an $m\times 1$ real-valued vector.
The Kaczmarz method solves the linear system $A\x = \b$ in the following way: Let the $i$-th row of $A$ be $A_i^{\TT}$, the $i$-th entry of $\b$ be $b_i$. Arbitrarily choose an initial approximation $\x_0$ of the solution. For $k=0,1,2,\cdots$, compute
\be \label{eq1}
\x_{k+1} = \x_k + \alpha_k   \frac{b_{i_k} -A_{i_k}^{\TT}\x_k}{\|A_{i_k}\|^2}   A_{i_k} ,
\ee
where $i_k\in\{1,2,\ldots,m\}$ is chosen according to some predefined principles and $\alpha_k\in(0,2)$ is the {\em relaxation parameter} and $\|\cdot\|$ denotes the 2-norm. In each iteration, the Kaczmarz method only uses one row of the matrix $A$, which makes it easy to implement. When the relaxation parameter $\alpha_k = 1$, the Kaczmarz method has a clear geometric meaning. It returns the solution of $A_{i_k}^{\TT} \x = \b_{i_k}$ that has the minimal distance to $\x_k$. Namely, the approximate solution $\x_{k+1}$ discovered in the $k$-th step is the orthogonal projection of $\x_k$ on the hyperplane defined by $A_{i_k}^{\TT} \x = \b_{i_k}$.

This method was first discovered in 1937 by Kaczmarz  \cite{karczmarz1937angenaherte} and was rediscovered in 1970 by Gordon, Bender, and Herman \cite{GORDON1970471} in the field of image reconstruction. In the most original form, $\alpha_k = 1$ and $i_k = (k \mod m) + 1$. Regarding this method, the conditions for convergence have been established, while useful theoretical estimates of the rate of convergence are difficult to obtain.  Previous results \cite{Deutsch1985,DEUTSCH1997381,GALANTAI200530,ma2015convergence} relate to some quantities (e.g., $\det(A^{\TT}A)$) depending on $A$ which are usually hard to use as a criterion to compare the performance of the Kaczmarz method with other iterative methods. 

A popular way to resolve this is to introduce randomness into the Kaczmarz method. In 2009, Strohmer and Vershynin \cite{strohmer2009randomized} first introduced a randomized version of the Kaczmarz method, in which $\alpha_k = 1$ and $i_k$ is chosen according to the probability distribution that ${\rm Prob}(j) = \|A_j\|^2/\|A\|_F^2$ for $j\in\{1,2,\ldots,m\}$. They gave a tight estimate of the convergence rate (i.e., \cs{the number of iterations}). The expected \cs{number of iterations} quadratically depends on the scaled condition number of $A$.
They also provided evidence that, in some cases, their randomized Kaczmarz method is more efficient than the conjugate gradient method, the most popular algorithm for solving large linear systems. Later in 2015, Gower and Richt{\'a}rik \cite{gower2015randomized} developed a versatile randomized iterative method for solving linear equations. \cs{It includes Strohmer and Vershynin's algorithm as a special case.}

To further \cs{improve the efficiency} of the randomized Kaczmarz method, the idea of parallelism was used, such as in \cite{richtarik2020stochastic,necoara2019faster,moorman2021randomized}. The basic idea is to use multiple rows in each step of the iteration. \cs{This will increase the cost of each step of the iteration. But it can reduce the number of iterations as expected.} One simple version (known as randomized  block Kaczmarz method \cite{moorman2021randomized}) is 
\be \label{neq1}
\x_{k+1} = \x_k + \alpha_k \sum_{i\in Q_k} w_i \frac{b_{i} -A_{i}^{\TT}\x_k}{\|A_{i}\|^2}  A_{i} ,
\ee
where the weights $w_i\in[0,1]$ satisfy $\sum_i w_i = 1$, and $\alpha_k \in (0,2)$. The multiset $Q_k$ is a collection of row indices such that $i$ is put into it with probability  $\|A_i\|^2/\|A\|_F^2$. 
It was shown (e.g., in \cite{moorman2021randomized}) that  \cs{the number of iterations} of this randomized  block Kaczmarz method is quadratic in the condition number of $A$ \cs{when  $\#(Q_k)= \lceil \|A\|_F^2/\|A\|^2 \rceil$. Although the overall complexity does not change theoretically, it usually works very well in practice.} This method was recently used in \cite{shao-montanaro21} to give the currently best quantum-inspired classical algorithm for solving linear equations. \cs{In addition to} the above work, there are many  studies aimed at accelerating or generalizing the (randomized) Kaczmarz method, for example, see \cite{eldar2011acceleration,gower2019adaptive,jiao2017preasymptotic,needell2014paved,needell2014stochastic,zouzias2013randomized,needell2015randomized,yaniv2021selectable,jarman2021randomized,xiang2017randomized,wu2020projected,shao2020row}.

\subsection{Our results}

In this paper, we investigate the Kaczmarz method (\ref{eq1}) from a new aspect. Firstly, instead of using $\alpha_k \in (0,2)$, we  choose $\alpha_k = 2$. This will generate a series of vectors $\{\x_0,\x_1,\x_2,\ldots\}$ distributed on a sphere centered on a solution. Apparently, $\x_k$ does not converge to any solution of the linear system $A\x=\b$ even if it is consistent. To obtain a solution, we study the  problem of determining $N$ and ${k_1},\ldots,{k_N}$ such that the average $(1/N)\sum_{j=1}^{N} \x_{k_j}$ is close to a solution. \cs{In this process, we find a new geometric fact of the solutions of linear systems.}
Secondly, we will not introduce randomness. To be more exact, we still set $i_k = (k\mod m) + 1$ in the $k$-th step of the iteration. As a direct result, this Kaczmarz method is deterministic. 

Throughout this paper, the following notation will be used frequently
\bea
\R_i &:=& I_n - 2 \frac{A_{i}A_{i}^{\TT}}{\|A_{i}\|^2}, \quad (i=1,2,\ldots,m),  \label{reflection} \\
\R_A &:=& \R_m \cdots \R_2 \R_1. \label{prod of reflections}
\eea
Obviously, $\R_i$ is the reflection generated by the $i$-th row of $A$.
With this notation, \cs{the iterative process} (\ref{eq1}) with $\alpha_k=2$ can be simplified as
\be \label{eq1-0}
\x_{k+1} = \R_{i_k} \x_k + \frac{2 b_{i_k}}{\|A_{i_k}\|^2}   A_{i_k}.
\ee

\subsubsection{The matrix $A$ is invertible}

When $A$ is invertible, the linear system $A\x=\b$ has a unique solution $\x_* = A^{-1} \b$. In this case, the iterative procedure (\ref{eq1-0}) can be reformulated as
\be \label{eq1-case1}
\x_{k+1} = \x_* + \R_{i_k} (\x_k - \x_*).
\ee
Since $\R_{i_k}$ is a reflection, all $\x_k$ lie on the sphere
\be
\S := \{\x \in \mathbb{R}^n: \|\x - \x_*\| = \|\x_0 - \x_*\|\}.
\ee
We consider the following two subsets of $\S$
\bea 
\S_0 \hspace{.3cm} 
:=& \hspace{-.34cm} 
\{\y_i:=\x_* + \R_A^{2i} (\x_0-\x_*): i = 0,1,2,\ldots\} 
& = \hspace{.3cm}  
\{\x_0, \x_{2n}, \x_{4n}, \ldots \}, \label{sphere 0} \\
\S_1 \hspace{.3cm}
:=& \{\z_i:=\x_* + \R_A^{2i+1} (\x_0-\x_*): i = 0,1,2,\ldots\} 
&= \hspace{.3cm}  
\{\x_{n}, \x_{3n}, \x_{5n}, \ldots \}.
\label{sphere 1}
\eea
In the following, the sphere with minimal dimension such that $\S_j$ lies on it is called the {\em minimal sphere} supporting $\S_j$. 
Our first main result is summarized as follows.

\begin{repthm}{main thm}
For $j=0,1$,  let $\c_j$ be the center of the minimal sphere supporting $\S_j$, then $\c_0+\c_1 = 2\x_*$. 
\end{repthm}

Theorem \ref{main thm} provides a clear geometric description of solutions of linear systems. The midpoint of the centers of two spheres generated by the process (\ref{eq1-0}) is exactly the solution of the linear system, see Figure \ref{fig3-2} for an illustration of dimension 3. This reduces the algebraic problem of solving a linear system of equations into a geometric problem of finding the centers of spheres. A simple way to find the centers is to take the average of some points on the spheres. This leads to the question, how many points do we need? Before presenting our next main result, we introduce the following concept, \csp{which has a close connection to the condition number of $A$, see Propositions \ref{prop:upper bound of convergence rate} and \ref{prop:relation of eta and kappa}.}

\begin{figure}[h]
     \centering
     \subfigure[An illustration of the distribution of $\S_0,\S_1$ on the sphere: the two circles represent the first 1000 vectors in $\S_0, \S_1$, the red points are their centers. The blue point is the solution, which is also the center of the sphere.]{
     \includegraphics[width=5.5cm]{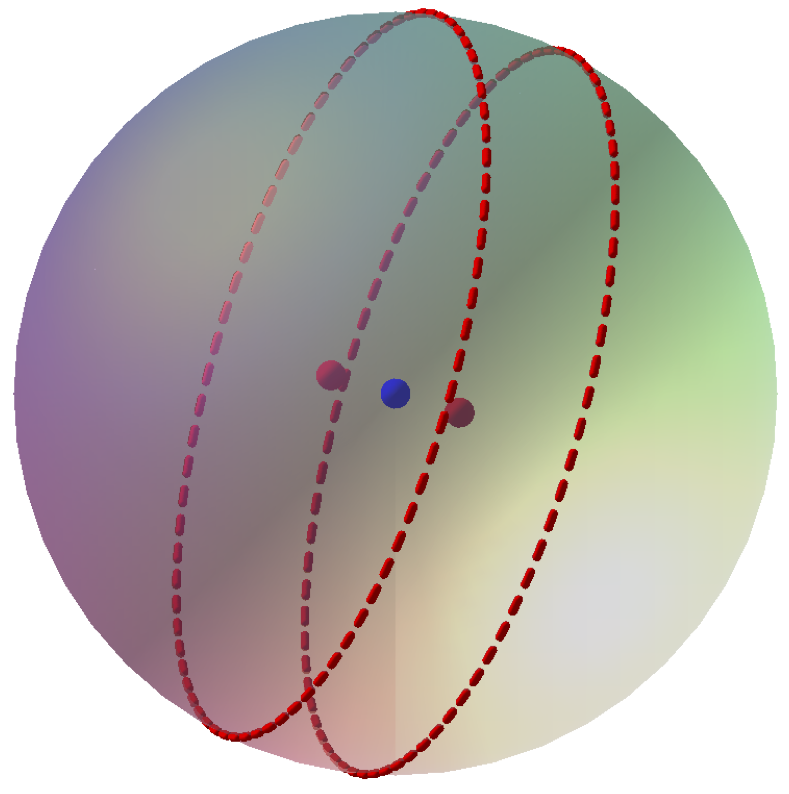}
     \label{fig3-2}
     }
     \hspace{2cm}
     \subfigure[The green line is the set of solutions, the red points are the centers of the blue circles generated by \cs{the procedure (\ref{eq1-case1})}.]{
     \includegraphics[width=5cm]{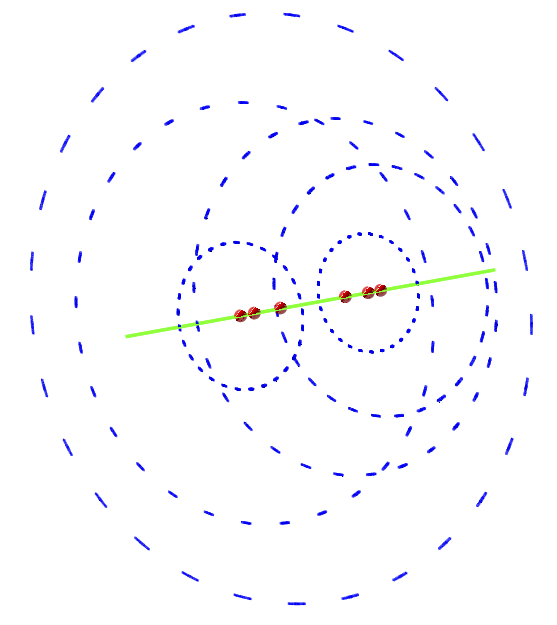}
     \label{circles in the singular case}
     }
     \caption{Illustrations of Theorems \ref{main thm} and \ref{thm1:singular case}.}
\end{figure}

{

\begin{defn}[Eigengap inverse]
\label{defn:eta}
Let $A$ be an $m\times n$ matrix with no zero rows. Let $\R_A$ be given in (\ref{prod of reflections}).
Denote the eigenvalues of $\R_A$ as $\{e^{i\theta_1},\ldots,e^{i\theta_n}\}$, where $|\theta_j|\leq \pi$. Then we define
\be
\eta(A) := \frac{1}{\min_{l: \theta_l\neq 0} |\theta_l|},
\ee
and call it the eigengap inverse of $\R_A$ (or $A$ for simplicity).
\end{defn}
}

\begin{repthm}{main thm2}
\cs{Let $A$ be an $n\times n$ invertible matrix, $\x_0, \b$ be two $n\times 1$ vectors, and $\x_* = A^{-1}\b$. Let $\y_j,\z_j$ be given in (\ref{sphere 0}), (\ref{sphere 1}), then}
\be
\label{main thm2:eq11}
\|\frac{1}{2N} \sum_{i=0}^{N-1} (\y_i+\z_i) - \x_*\| \leq \varepsilon \|\x_0-\x_*\|,
\ee
where
{
\be
N = \left\lceil \frac{\pi \eta(A)}{2\varepsilon } \right\rceil.
\ee
}
\end{repthm}

\cs{Based on the above result, to obtain an approximation of the solution up to relative error $\varepsilon$, we need to generate $O(\eta(A)/\varepsilon)$ points on the sphere. Indeed, this can be reduced to $O(\eta(A)\log(1/\varepsilon))$ by setting $\varepsilon=1/2$ in (\ref{main thm2:eq11}), viewing $\frac{1}{2N} \sum_{i=0}^{N-1} (\y_i+\z_i)$ as the new approximation of the solution, and restarting the procedure.\footnote{\csp{We would like to thank an anonymous referee for pointing out this idea, which is from Steinberger's paper \cite{steinerberger2020surrounding}, a paper closely related to ours. We will discuss more on this in Subsection 1.3.}}  
In this paper, the algorithm for solving linear systems based on (\ref{main thm2:eq11}) will be called the {\em deterministic iterative reflection algorithm}.
An explicit description of this algorithm is shown in Algorithm \ref{Kaczmarz algorithm}.}

\subsubsection{General consistent linear systems}

We now turn to the general case. For consistent linear systems, there can be one or infinitely many solutions. In the latter case, the \cs{deterministic iterative reflection algorithm} demonstrates some interesting properties. The following result states that if we run the \cs{deterministic iterative reflection algorithm} from $\x_0$, we end up with the solution that has the minimal distance to $\x_0$, see Figure \ref{circles in the singular case} for an illustration in dimension 3. \cs{In Figure \ref{circles in the singular case}, each circle is generated by the deterministic iterative reflection algorithm from some $\x_0$. The center of each circle is a solution of the linear system.}

\begin{repthm}{thm1:singular case}
Assume that $A\x = \b$ is consistent.
Let $\x_0$ be an arbitrarily chosen initial vector, then $\{\x_0,\x_1,\x_2,\ldots\}$ \cs{generated by the process (\ref{eq1-0})} lies on a sphere 
centered on the solution that has the minimal distance to $\x_0$.
\end{repthm}

\cs{Regarding Theorems \ref{main thm}, \ref{main thm2}, we prove that they are also correct if $A$ is reflection consistent (see Definition \ref{defn:Reflection consistent}). Reflection consistency is almost equivalent to the claim that $m-r$ is even, where $r$ is the rank of $A$, see Corollary \ref{cor:reflection consistence}. Although the condition that $m-r$ is even is easy to satisfy, it is interesting to see that the algorithm requires this special condition.}
To satisfy this condition, a simple approach is to randomly introduce some new linear constraints that do not change the solution set of $A\x = \b$. A difficulty here is that \cs{not all linear constraints are allowed}. We will show that the set of the new linear constraints such that the \cs{deterministic iterative reflection algorithm} fails is a set with zero Lebesgue measure. So we can choose almost any linear constraints we want.

Finally, when the linear system $A\x=\b$ is inconsistent, the \cs{deterministic iterative reflection algorithm} fails to find the least-squares solution. We show that the \cs{deterministic iterative reflection algorithm} indeed solves the linear system $A^{\TT}W^{-1}A\x = A^{\TT}W^{-1} \b$, where $W$ is the lower-triangular matrix such that $W+W^{\TT}=2AA^{\TT}$. \cs{Since solving least-squares problem $\min_{\x}\|A\x-\b\|$ is equivalent to solving $A^{\TT}A\x = A^{\TT} \b$, the above finding explains why the deterministic iterative reflection algorithm fails to return the least-squares solution.}
Note that the new linear system has a similar structure to the generalized least-square problem in statistics \cite{GoluVanl96}, where $W$ becomes a covariance matrix. 
For consistent linear systems, the new linear system $A^{\TT}W^{-1}A\x = A^{\TT}W^{-1} \b$ is equivalent to the original one when $A$ is reflection consistent. 
We will show that under reasonable assumptions, which can be satisfied easily by the \cs{idea} discussed in the above paragraph for consistent linear systems, the matrix $A^{\TT}W^{-1}A$ is invertible if $A$ has full column-rank. In this case, the new linear system $A^{\TT}W^{-1}A\x = A^{\TT}W^{-1} \b$ has a unique solution. Moreover, similar theoretical guarantees to Theorems \ref{main thm}, \ref{main thm2} also hold.
\cs{Nevertheless, it still remains a problem to modify the deterministic iterative reflection algorithm to solve the least-squares problems.}

\subsection{Comparison with previous Kaczmarz algorithms}

In this part, we theoretically compare the deterministic iterative reflection algorithm with two well-studied randomized Kaczmarz algorithms \cite{strohmer2009randomized, moorman2021randomized}. The  numerical comparisons are given in Section \ref{section:Numerical experiments}.

From Equations (\ref{sphere 0}), (\ref{sphere 1}), we know that $\y_{i}=\x_{2im}$ and $\z_{i}=\x_{(2i+1)m}$. Consequently, we can compute $\y_i$ (respectively $\z_i$) from $\y_{i-1}$ (respectively $\z_{i-1}$) with $O(mn)$ operations. By Theorem \ref{main thm2}, it totally uses $O(mn\eta(A)\log(1/\varepsilon))$ operations to obtain an approximation of the solution up to relative error $\varepsilon$.

Compared to previous Kaczmarz methods \cite{strohmer2009randomized, moorman2021randomized}, it is costly to use $O(mn)$ operations in each step of the iteration. 
There is actually a simple approach to resolving this problem.
We can decompose the set $\mathbb{X}:=\{\x_0,\x_1,\x_2, \ldots\}$ generated by the procedure (\ref{eq1-0}) into a union of $2m$ subsets, which have similar structures to $\S_0,\S_1$. More precisely, for $k\in\{0,1,\ldots,m-1\}$ and $j \in\{0,1\}$, we define
\bea
\label{new-def-intro}
\S_{kj} := \{ \x_* + \R_k \cdots \R_1 \R_A^{2i+j} (\x_0-\x_*): i = 0,1,2,\ldots\} 
= \{ \x_{k+(2i+j)m}: i = 0,1,2,\ldots\}, 
\eea
then
$\mathbb{X} =  \cup_{k=0}^{m-1} \cup_{j=0}^1 \S_{kj}$. As a direct consequence, we have similar results to Theorems \ref{main thm}, \ref{main thm2} for each $\S_{k0}\cup \S_{k1}$. Hence, we can use the average of the first $O(m\eta(A)\log(1/\epsilon))$ points in $\mathbb{X}$ to approximate the solution. This algorithm is summarized in Algorithm \ref{Kaczmarz algorithm2}. It requires more steps of iterations to converge but has a lower computational cost in each step of the iteration. Theoretically, this algorithm has the same complexity as the previous one. 

\setlength{\arrayrulewidth}{0.3mm}
{\renewcommand
\arraystretch{1.5}
\begin{table}[h]
\centering
\begin{tabular}{|c|c|c|c|c|} 
 \hline
 Algorithm & \# operations in each step & \# iterations &  Type \\ \hline
  SV \cite{strohmer2009randomized} & $O(n)$ & $O({\kappa_s^2(A)}\log(1/\varepsilon))$ & Randomized  \\
  MTMN \cite{moorman2021randomized} & $O(n\|A\|_F^2/\|A\|^2)$ & $O({\kappa^2(A)}\log(1/\varepsilon))$ & Randomized block \\
  Alg. \ref{Kaczmarz algorithm} (This paper) & \cs{$O(mn)$} & $O({\eta(A)\log (1/\varepsilon )})$
 & Deterministic  \\
  Alg. \ref{Kaczmarz algorithm2} (This paper) & $O(n)$ & $O({m\eta(A)\log (1/\varepsilon)})$
 & Deterministic  \\
 \hline
\end{tabular}
\caption{Comparison of different Kaczmarz methods, {where $\kappa_s(A)=\|A\|_F\|A^{-1}\|$ is the scaled condition number, $\kappa(A)=\|A\|\|A^{-1}\|$ is the condition number, and $\eta(A)$ is given in Definition \ref{defn:eta}.}}
\label{table:comparison}
\end{table}
}

In Table \ref{table:comparison}, we compare four randomized Kaczmarz algorithms in terms of the number of operations in each step and the total number of iterations.
\begin{enumerate}
\item Different from the other two algorithms, Algorithms \ref{Kaczmarz algorithm}, \ref{Kaczmarz algorithm2} are influenced by the eigengap inverse $\eta(A)$ arising from the eigenvalues of $\R_A$. It is generally a quite hard problem to estimate $\eta(A)$. In this paper, we show that $\eta(A) = O((\log n) \kappa^2(A))$ in the worst case and $\eta(A) = O(\kappa(A))$ for random matrices when $m\geq n\log n$; see Propositions \ref{prop:upper bound of convergence rate}, \ref{prop:relation of eta and kappa}.
%
%
\item The overall cost of SV's and MTMN's algorithm is $O(n\kappa_s^2(A)\log(1/\varepsilon))$, and the overall cost of Algorithm \ref{Kaczmarz algorithm} is $O(mn\eta(A) \log(1/\varepsilon))$.
Note that $\kappa_s(A) \leq \min(m,n)\kappa(A)$. So for over-determined linear systems (i.e., $m> n$),  Algorithm \ref{Kaczmarz algorithm} seems worse even if $\eta(A) = O(\kappa(A))$. This seems reasonable as usually $O(n)$ rows of $A$ are enough to solve over-determined linear systems. 
\end{enumerate}

The above theoretical results indicate that Algorithm \ref{Kaczmarz algorithm} is usually less efficient in practice. This is not surprising since it uses all the rows of the input matrix when computing a new useful vector. When the linear system is over-determined, this is undesirable. To overcome this problem, in practice, we can modify this algorithm by first computing $\x_0,\ldots,\x_{M-1}$ and then checking the quality of their average as an approximate solution. If it is not a good approximation, then we can restart this process from the average. From our numerical tests in Section \ref{section:Numerical experiments}, this modified algorithm can be faster than SV's algorithm and can compete with randomized block Kaczmarz algorithms.

\subsection{Related work}

There are many research papers on the Kaczmarz method, we list below the most related ones.
A similar idea was recently used by Steinerberger. In \cite{steinerberger2020surrounding}, Steinerberger focused on the solving of nonsingular square linear systems using the randomized Kaczmarz method by setting $\alpha_k = 2$ as well. The randomness is similar to that of Strohmer-Vershynin's algorithm. Steinerberger proved that 
\be
\label{Steinerberger-eq1}
\E
\left\|\frac{1}{N} \sum_{j=0}^{N-1} \x_j - A^{-1}\b \right\| \leq \varepsilon \|\x_0 - A^{-1}\b\|
\ee
when $N=O(\kappa_s^2(A)/\varepsilon^2)$. So by Markov’s inequality, with high probability the average of the first $N$ vectors can be viewed as an approximation of the solution.   Note that in our second algorithm, we have a similar estimate
\be
\label{Steinerberger-eq2}
\left\|\frac{1}{N} \sum_{j=0}^{N-1} \x_j - A^{-1}\b \right\| \leq \varepsilon \|\x_0 - A^{-1}\b\|
\ee
when $N=O(n\eta(A)/\varepsilon)$. The $\x_j$'s in (\ref{Steinerberger-eq1}), (\ref{Steinerberger-eq2}) have different meanings. Different from Steinerberger's result, our result is deterministic.
In \cite{steinerberger2020surrounding}, Steinerberger posed an open question of finding a better and deterministic way to approximate the solution from the samples. Our algorithm can be viewed as an answer to this open question -- we sample points in a deterministic way and sub-sample smartly so that the samples have predictable structures on the sphere and their average is unbiased with respect to the solution regardless of the linear system used for creating the samples. 

Choosing $\alpha_k=2$ in (\ref{eq1}) is not new. It was highlighted in another famous row projection method proposed by Cimmino in 1938 \cite{Cimmino}. In Cimmino's method, from $\x_k$, we perform $m$ reflections (\ref{eq1-0}) for $i_k\in\{1,2,\ldots,m\}$ and use their average to define $\x_{k+1}$. It \cs{was} shown by Cimmino that the sequence $\{\x_0,\x_1,\ldots\}$ converges to a solution under the mild assumption that ${\rm Rank}(A)\geq 2$. In our algorithm, $\x_{k+1}$ is obtained by one reflection and the sequence $\{\x_0,\x_1,\ldots\}$ locates on a sphere centered on a solution. \cs{The} Cimmino method is known to be more amenable to parallelism than \cs{the} Kaczmarz method. However, the required number of iterations for Cimmino's method could be large. Compared to Equation (\ref{neq1}), \cs{the randomized block} Kaczmarz method seems to be a generalization of the Cimmino method, except that we are now allowed to set $\alpha_k = 2$. For more on the connection between the Kaczmarz method and the Cimmino method, we refer to \cite{ansorge1984connections,benzi2004gianfranco}.

The idea of using cyclic subsequence $\{\x_{jm+k}:j=0,1,\ldots,k=0,1,\ldots,m-1\}$ in the Kaczmarz method to solve inconsistent linear systems has been studied decades ago, for example see \cite{tanabe1971projection,censor1983strong,eggermont1981iterative}. In those papers, the authors investigated the problem of using what kind of relaxation parameters the cyclic subsequences will converge to the least-square solution. In \cite{eggermont1981iterative},  Eggermont,  Herman, and Lent proved that if the relaxation parameters are periodic, then the cyclic subsequences converge to the least-square solution. It was  shown in \cite{censor1983strong} that when the relaxation parameters tend to zero, the limits of the cyclic subsequences approach the least-square solution. 

The original Kaczmarz algorithm \cite{karczmarz1937angenaherte} is deterministic. There are also some other versions of deterministic Kaczmarz algorithms. We name a few here. \csp{For more, we refer to \cite{petra2016single,niu2020greedy,censor1981row,feichtinger1992new} and the references therein.} In 1954, Agmon \cite{agmon1954relaxation},  Motzkin and Schoenberg \cite{motzkin1954relaxation} extended the Kaczmarz algorithm to solve linear inequalities by orthogonally projecting the current solution onto the chosen halfspace. In 1957, Hildreth \cite{hildreth1957quadratic} also proposed a similar deterministic algorithm to solve linear inequalities to find the closest point in the solution set to a given point. It is worth mentioning that Hildreth's algorithm can be reduced to the original Kaczmarz algorithm. In \cite{chen2021fast}, Chen and Huang proposed a deterministic block Kaczmarz method for solving the least-squares problem which is competitive with randomized block Kaczmarz methods. \csp{In \cite{nutini2016convergence}, Nutini et al introduced two greedy selection rules that make the Kaczmarz method deterministic. The greedy selection rules give faster convergence rates, and the costs are similar to the randomized ones in some applications.}

\subsection{Outline of this paper}

In Section \ref{section:Preliminaries}, we prove some lemmas that will be used in our proofs of the main theorems.
In Section \ref{section:Consistent linear systems}, we prove our main results for consistent linear systems.
In Section \ref{section:Inconsistent linear systems}, we focus on inconsistent linear systems and investigate deeper on the Kaczmarz method.
Finally, in Section \ref{section:Numerical experiments} we compare different Kaczmarz methods numerically. 

\section{Preliminaries}
\label{section:Preliminaries}

Throughout this paper, we use $\{\e_1,\ldots,\e_n\}$ to denote the standard basis of $\mathbb{R}^n$, i.e., for $\e_j$, the $j$-th entry is 1 and all other entries are 0. The $n\times n$ identity matrix will be denoted as $I_n$. For the linear system $A\x = \b$, we always assume that $A$ has no zero rows. All vectors that appear in this paper are assumed to be given in column forms. So when we say the $i$-th row of $A$, we shall use $A_i^{\TT}$ because $A_i$ refers to a column vector. {The operator norm of $A$ is denoted as $\|A\|$. It is the maximal singular value of $A$.  With $\|A\|_F$, we mean the Frobenius norm, which is the square root of the sum of the absolute squares of the elements of $A$. The condition number refers to the ratio of the largest singular value to the smallest nonzero singular value. For any two vectors $\a,\b$, their inner product is denoted as $\langle \a|\b\rangle$. The transpose of $A$ is denoted as $A^\TT$.} We remark again that the notation (\ref{reflection}), (\ref{prod of reflections}) will be used frequently in this paper. 

In this section, we aim to prove some preliminary lemmas that will be used in the next section.
The following two lemmas are useful in the proofs of Theorems \ref{main thm} and \ref{main thm2}.

\begin{lem}
\label{key lemma}
Let $\theta_1,\ldots,\theta_p \in (0,2\pi)$ be $p$ distinct parameters and $\{\v_1,\ldots,\v_p \} \subseteq \mathbb{R}^n$ be a set of orthogonal vectors, then the dimension of the vector space spanned by $\{\sum_{k=1}^p (1-e^{ij\theta_k }) \v_k:j\in \mathbb{N}\}$ is $p$. Moreover, they lie on a sphere centered at $\sum_{k=1}^p \v_k$.
\end{lem}

\begin{proof}
We consider the following matrix with $p$ columns:
\[
M := \begin{pmatrix}
1-e^{i\theta_1} & 1-e^{2i\theta_1} & 1-e^{3i\theta_1} & \cdots  \\
\vdots & \vdots  & \vdots & \vdots \\
1-e^{i\theta_p} & 1-e^{2i\theta_p} & 1-e^{3i\theta_p} & \cdots
\end{pmatrix}.
\]
We aim to prove that $M$ is nonsingular.

Since $1-e^{ij\theta_k} = (1-e^{i\theta_k}) (1+e^{i\theta_k}+\cdots+e^{i(j-1)\theta_k})$ and $e^{i\theta_k}\neq 1$, it follows that the rank of $M$ is equal to the rank of the following matrix, which is obtained by dividing the $k$-th row of $M$ with $(1-e^{i\theta_k})$:
\[
\begin{pmatrix}
1 & 1+e^{i\theta_1} & 1+e^{i\theta_1}+e^{2i\theta_1} & \cdots \\
\vdots & \vdots & \vdots & \vdots\\
1 & 1+e^{i\theta_p} & 1+e^{i\theta_p}+e^{2i\theta_p} & \cdots
\end{pmatrix}.
\]
Via some column operations, which do not change the rank, the above matrix can be transformed to
\[
\begin{pmatrix}
1 & 0 & 0 & \cdots \\
1 & e^{i\theta_2}-e^{i\theta_1} & e^{2i\theta_2}-e^{2i\theta_1} & \cdots \\
\vdots & \vdots & \vdots & \vdots\\
1 & e^{i\theta_p}-e^{i\theta_1} & e^{2i\theta_p}-e^{2i\theta_1} & \cdots
\end{pmatrix}.
\]

Now the singularity of $M$ is equivalent to the singularity of the following matrix with $(p-1)$ columns
\[
\begin{pmatrix}
e^{i\theta_2}-e^{i\theta_1} & e^{2i\theta_2}-e^{2i\theta_1} & \cdots \\
\vdots & \vdots & \vdots\\
e^{i\theta_p}-e^{i\theta_1} & e^{2i\theta_p}-e^{2i\theta_1} & \cdots
\end{pmatrix}.
\]
Dividing the $j$-th row with $(e^{i\theta_j}-e^{i\theta_1})$, which is not zero by assumption, we obtain the following matrix
\[
\begin{pmatrix}
1 & e^{i\theta_2}+e^{i\theta_1} & e^{2i\theta_2}+e^{i\theta_2}e^{i\theta_1}+e^{2i\theta_1} & \cdots \\
\vdots & \vdots & \vdots & \vdots \\
1 & e^{i\theta_p}+e^{i\theta_1} & e^{2i\theta_p}+e^{i\theta_p}e^{i\theta_1}+e^{2i\theta_1} & \cdots
\end{pmatrix}
\xrightarrow[]{\text{Column operations}}
\begin{pmatrix}
1 & e^{i\theta_2} & e^{2i\theta_2} & \cdots \\
\vdots & \vdots & \vdots & \vdots \\
1 & e^{i\theta_p} & e^{2i\theta_p} & \cdots
\end{pmatrix}.
\]
Finally, the claimed result about the dimension follows directly by induction on $p$. The second claim is straightforward.
\end{proof}

The following result is a generalization of Brady-Watt's formula \cite{brady2006products}. Although they only considered the full row-rank case, their result is true generally. Below, we give a simple proof by induction.

\begin{lem}[Brady-Watt's formula]
\label{lem:Brady-Watt formula}
Let $A$ be an $m\times n$ matrix with no zero rows, \cs{$\R_A$ be given in (\ref{prod of reflections})}, then
\[
\R_A = I_n - 2 A^{\TT}  W^{-1} A,
\]
where $W$ is the lower-triangular matrix such that $W+W^{\TT} = 2 AA^{\TT}$.
\end{lem}

\begin{proof}
We prove it by induction on $m$.
When $m=1$, the result is clearly true. Now assume that
\[\R_A = I_n - 2 A^{\TT}  W_1^{-1} A, \quad
\R_B = I_n - 2 B^{\TT}  W_2^{-1} B,
\]
where $W_1,W_2$ are lower-triangular matrices such that $W_1+W_1^{\TT} = 2 AA^{\TT}$ and $W_2+W_2^{\TT} = 2 BB^{\TT}$. Let $W$ be the lower-triangular matrix such that
\[
W+W^{\TT} = 2 
\begin{pmatrix}
A \\
B 
\end{pmatrix}
\begin{pmatrix}
A^{\TT} & B^{\TT}
\end{pmatrix}
= \begin{pmatrix}
2AA^{\TT} & 2AB^{\TT} \\
2BA^{\TT} & 2BB^{\TT}
\end{pmatrix}
= \begin{pmatrix}
W_1+W_1^{\TT} & 2AB^{\TT} \\
2BA^{\TT} & W_2+W_2^{\TT} 
\end{pmatrix}.
\]
Namely,
\[
W = \begin{pmatrix}
W_1 & 0 \\
2BA^{\TT} & W_2
\end{pmatrix}.
\]
It is easy to show that
\[
W^{-1} = \begin{pmatrix}
W_1^{-1} & 0 \\
-2W_2^{-1}BA^{\TT}W_1^{-1} & W_2^{-1}
\end{pmatrix}.
\]
Denote $C = \begin{pmatrix}
A \\
B 
\end{pmatrix}$.
We then can check that $\R_{C} =
\R_B \R_A = I_n - 2 C^{\TT} W^{-1} C$, where $W$ is the lower-triangular matrix such that $W+W^{\TT} = 2 CC^{\TT}$.
Finally, the claimed result follows by induction.
\end{proof}

As a corollary of Brady-Watt's formula, we have the following result. We below present a new proof, which will be helpful in the estimate of the convergence rate in the next section.

\begin{lem}
\label{key lemma 2}
Assume that $A$ is an $n\times n$ matrix without zero rows, \cs{ and $\R_A$ is given in (\ref{prod of reflections})}, then
\[
\det(\R_A-I_n)=(-2)^n \det(A)^2\prod_{i=1}^n \|A_i\|^{-2}.
\]
Consequently, 1 is not an eigenvalue of $\R_A$ if $A$ is nonsingular.
\end{lem}

\begin{proof}
From the definition of $\R_A$, it suffices to consider the case that each row of $A$ has a unit norm. We prove the result by induction on $n$. Assume that $A^{\TT}=Q L$ is the QL decomposition of $A^{\TT}$, where $Q$ is orthogonal and $L$ is lower triangular. Then the $i$-th row of $A$, when written in a column vector, equals $A_i = Q L_i$, where $L_i$ is the $i$-th column of $L$. As a result, we have
\[
\R_A = Q \prod_{i=1}^n (I_n - 2 L_i L_i^{\TT}) Q^{\TT} .
\]
The eigenvalues and the determinant of $\R_A$ are the same as those of $\prod_{i=1}^n (I - 2 L_i L_i^{\TT})$. Hence, without loss of generality, we can assume that $A$ is a lower triangular matrix satisfying that each row has a unit norm. 

When $n=1$, we have $A=1$ and $\R_A = -1$. Thus, $\det(\R_A - I_1) = -2$. When $n=2$,  
$A = \begin{pmatrix}
1 & 0 \\
a & b \\
\end{pmatrix}$ where $a^2+b^2=1$. Thus
$
\R_A = \begin{pmatrix}
2a^2-1 & -2a b \\
2ab    & 1-2b^2 \\
\end{pmatrix}
$. We now can directly check that $\det(\R_A-I_2)=4 b^2$. When $n\geq 3$, we denote the last row of $A$ as $(\v^{\TT}, w)$, where $\v \in \mathbb{R}^{n-1}$ is a column vector and $w\in \mathbb{R}^*$. For convenience, we denote the submatrix with the last row and column of $A$ removed as $B$. Then it is easy to show that
\bes
\R_A = \begin{pmatrix}
(I_{n-1}-2\v\v^{\TT})\R_B & -2w\v \\
-2w\v^{\TT} \R_B    & 1-2w^2 \\
\end{pmatrix}.
\ees
Therefore,
\beas
\R_A-I_n &=& \begin{pmatrix}
(I_{n-1}-2\v\v^{\TT})\R_B-I_{n-1} & -2w\v \\
-2w\v^{\TT} \R_B    & -2w^2 \\
\end{pmatrix} \\
&=&
\begin{pmatrix}
I_{n-1} & 0 \\
0 & w \\
\end{pmatrix}
\begin{pmatrix}
(I_{n-1}-2\v\v^{\TT})\R_B-I_{n-1} & \v \\
-2\v^{\TT} \R_B    & 1 \\
\end{pmatrix}
\begin{pmatrix}
I_{n-1} & 0 \\
0 & -2w \\
\end{pmatrix}.
\eeas
It implies
\[
\det(\R_A - I_n) = -2w^2 \det\begin{pmatrix}
(I_{n-1}-2\v\v^{\TT})\R_B-I_{n-1} & \v \\
-2\v^{\TT} \R_B    & 1 \\
\end{pmatrix}
=-2w^2 \det(\R_B-I_{n-1}).
\]
By induction, we have $\det(\R_A-I_n)=(-2)^n \det(A)^2$.
\end{proof}

The next lemma will be used to estimate an upper bound of $\eta(A)$ in terms of the condition number of $A$. The triangular truncation operator $T$ is defined as an $m\times m$ matrix such that $T_{ij}=1$ if $i\geq j$ and 0 otherwise. Denote
\[
K_m := \max \left\{ \frac{\|A \circ T\|}{\|A\|}: A \text{ is } m\times m \text{ and nonzero}, \circ \text{ is the Hadamard product}\right\}.
\]

\begin{lem}[Theorem 1 of \cite{angelos1992triangular}]
\label{lem: triangular truncation operator}
For $m\geq 2$,
\[
\left|\frac{K_m}{\log m} - \frac{1}{\pi}\right| \leq \left(1+\frac{1}{\pi}\right) \frac{1}{\log m}.
\]
\end{lem}
The above result suggests that $\|A\circ T\|=O( \|A\| \log m)$ for any $m\times m$ matrix $A$.

\section{Consistent linear systems}
\label{section:Consistent linear systems}

We now turn our attention to the main results of solving linear systems $A \x = \b$ by the Kaczmarz method. \cs{First, we explicitly describe our main  algorithm for solving consistent linear systems in Algorithm \ref{Kaczmarz algorithm} below}. Then, we analyze its correctness and efficiency. \cs{To better understand the geometric structures of the algorithm}, we first focus on the case that $A$ is invertible. We then extend the results to general consistent linear systems. In this section, the notation (\ref{reflection}), (\ref{prod of reflections}), (\ref{sphere 0}), (\ref{sphere 1}) will be widely used. 

\begin{algorithm}[H]
\caption{{Deterministic iterative reflection algorithm for solving consistent linear systems}}
\label{Kaczmarz algorithm}
\begin{algorithmic}[1]
\Require A consistent linear system $A\x=\b$, an arbitrarily chosen initial vector $\x_0$ \cs{and an accuracy $\varepsilon\in[0,1]$}.

\Ensure \cs{Output $\tilde{\x}$ such that $\|A\tilde{\x} - \b\| \leq \varepsilon$.}



\State For $k=0,1,2,\cdots$, compute
\be \label{our procedure}
\x_{k+1} = \x_k + 2 \frac{b_{i_k} -A_{i_k}^{\TT}\x_k}{\|A_{i_k}\|^2}   A_{i_k},
\ee
where $i_k= (k\mod m) + 1$.

\State \cs{Stop the iteration at step $m(M-1)$ if
$\|A\tilde{\x} - \b\| \leq \varepsilon$, where $\tilde{\x} = \frac{1}{M} \sum_{j=0}^{M-1} \x_{jm}$.}

\State \cs{Output $\tilde{\x}$.}
\end{algorithmic}
\end{algorithm}

\begin{rem}
\cs{
Before presenting all the results, we want to emphasize in a less rigorous way that Algorithm \ref{Kaczmarz algorithm} only works for consistent linear systems under the mild assumption that $m-r$ is even, where $r$ is the rank of $A$. 
}
\end{rem}


\subsection{Case 1: $A$ is invertible}

\cs{When $A$ is invertible, the solution is unique. This greatly simplifies the analysis of Algorithm \ref{Kaczmarz algorithm}. The geometric feature of the algorithm is also very clear in this case. Moreover, this case is a theoretical building block for the general case.}

Note that  $m=n$ in this case. As discussed in the introduction (see  (\ref{new-def-intro})),
in (\ref{sphere 0}), (\ref{sphere 1}), we know that $\S_0,\S_1$ are only two subsets of $ \mathbb{X} = \{\x_l:l=0,1,2,\ldots\}$. We can decompose $\mathbb{X}$ into a union of $2n$ subsets, which have similar structures to $\S_0,\S_1$. More precisely, for $k=0,1,\ldots,n-1$ and $j = 0,1$, define
\bea
\S_{kj} &:=& \{ \x_* + \R_k \cdots \R_1 \R^{2i+j} (\x_0-\x_*): i = 0,1,2,\ldots\} \\
&=& \{ \x_* + (\R_k\cdots \R_1 \R_n\cdots \R_{k+1})^{2i+j} (\x_k-\x_*): i = 0,1,2,\ldots\} \label{sphere 3}\\
&=& \{ \x_{k+(2i+j)n}: i = 0,1,2,\ldots\}, 
\eea
where $\x_k = \x_* + \R_k \cdots \R_1 (\x_0-\x_*)$. Hence we have
$\mathbb{X} =  \cup_{k=0}^{n-1} \cup_{j=0}^1 \S_{kj}$. This decomposition is a theoretical guarantee of our second algorithm below. \cs{Using this notation}, $\S_{0j}=\S_j$ for $j=0,1$. 
Recall that the sphere with minimal dimension such that $\S_j$ lies on it is called the {\em minimal sphere} supporting $\S_j$.

\begin{thm}
\label{main thm}
For $j=0,1$,  let $\c_j$ be the center of the minimal sphere supporting $\S_j$, then $\c_0+\c_1 = 2\x_*$. 
\end{thm}

\begin{proof}
Let $\R_A = U^\dag D U$ be the eigenvalue decomposition of $\R_A$, where $U$ is unitary, $D$ is diagonal, and `` $\dag$ " refers to the conjugate transpose operator.
For convenience, we assume that $\x_0 = 0$. If $\x_0\neq 0$, we can consider $\x_0 - \y_j, \x_0 - \z_j$ instead. This is just a translation, which will not affect the final result.
\cs{With the above notation}, we have $\y_j = U^\dag (I-D^{2j}) U \x_*$ and $\z_j = U^\dag (I-D^{2j+1}) U \x_*$.

Denote $\tilde{\y}_j = U \y_j, \tilde{\z}_j = U \z_j,  \tilde{\x}_* = U \x_* = \sum_{k=1}^n x_k \e_k$, and $D={\rm diag}(e^{i\theta_1},\ldots, e^{i\theta_n})$, \cs{where $|\theta_i| \leq \pi$ for all $i$.} Then
\[
\tilde{\y}_j = \sum_{k=1}^n (1-e^{2ij\theta_k }) x_k \e_k, 
\quad
\tilde{\z}_j = \sum_{k=1}^n (1-e^{i(2j+1)\theta_k }) x_k \e_k.
\]
For convenience, we denote all the distinct eigen-phases $\{\theta_1,\ldots, \theta_n\}$ of $\R_A$ as $\{\phi_1, \phi_2, \pm \phi_3,  \ldots, \pm \phi_L\}$, where $\phi_1=\pi, \phi_2=0$ and $0< \phi_l < \pi$ for $3 \leq l \leq L$. Certainly, it may happen that $\pi,0\not\in \{\theta_1,\ldots, \theta_n\}$. However, this will not influence the following analysis.

We first determine the center of the minimal sphere supporting $\S_0$. Note that
\beas
\tilde{\y}_j &=& \sum_{k=1}^L (1-e^{2ij\phi_k }) \sum_{l:\theta_l = \phi_k } x_l \e_l + \sum_{k=1}^L (1-e^{-2ij\phi_k }) \sum_{l:\theta_l = -\phi_k } x_l \e_l \\
&=& \sum_{k=3}^L \left(1-e^{ij (2\phi_k) } \right) \sum_{l:\theta_l = \phi_k } x_l \e_l 
+
\sum_{k=3}^L \left(1-e^{ij (2(\pi -\phi_k)) } \right) \sum_{l:\theta_l = -\phi_k } x_l \e_l.
\eeas
\cs{Without loss of generality,} we assume that $\sum_{l:\theta_l = \phi_k } x_l \e_l \neq 0, \sum_{l:\theta_l = -\phi_k } x_l \e_l \neq 0$ for all $k\geq 3$, otherwise we just ignore them. Since $0< 2\phi_k, 2(\pi -\phi_k) < 2 \pi$, it then follows from Lemma \ref{key lemma} that the dimension of the vector space spanned by $ \{\tilde{\y}_j:j=0,1,2,\ldots\}$ is $2L-4$. Moreover, the center of the minimal sphere that supports this vector space is 
\be \label{eq1 in proof thm1}
\sum_{k=3}^L ~ \sum_{l:\theta_l = \phi_k } x_l \e_l 
+
\sum_{k=3}^L ~ \sum_{l:\theta_l = -\phi_k } x_l \e_l
=
\tilde{\x}_* - \sum_{l:\theta_l = \pi } x_l \e_l - \sum_{l:\theta_l = 0 } x_l \e_l .
\ee

Regarding the center of the minimal sphere supporting $\S_1$, we notice that
\[
\tilde{\z}_j = \sum_{k=1}^n (1-e^{i\theta_k }) x_k \e_k
+ \sum_{k=1}^n (1-e^{2ij\theta_k }) e^{i\theta_k } x_k \e_k.
\]
The first term is a translation and the second term can be analyzed in a similar way to $\tilde{\y}_j$. Thus the dimension of the vector space spanned by $\S_1$ is also $2L-4$, and the center of the minimal sphere that supports $ \{\tilde{\z}_j:j=0,1,2,\ldots\}$ is 
\bea  \label{eq2 in proof thm1}
&& \sum_{k=1}^n (1-e^{i\theta_k }) x_k \e_k
+ \sum_{k=1}^n e^{i\theta_k } x_k \e_k
- \sum_{l:\theta_l = \pi } e^{i\theta_l } x_l \e_l - \sum_{l:\theta_l = 0 } e^{i\theta_l } x_l \e_l \nonumber \\
&=& \tilde{\x}_* + \sum_{l:\theta_l = \pi } x_l \e_l - \sum_{l:\theta_l = 0 } x_l \e_l.
\eea

By Lemma \ref{key lemma 2}, $\theta_l \neq 0$ for all $l$. Namely, the third term in (\ref{eq1 in proof thm1}), (\ref{eq2 in proof thm1}) does not exist.\footnote{The reason that we keep the third term until the end is to simplify the analysis when dealing with general consistent linear systems.} Therefore, we have
\[
\c_0 = U^\dag \Big(\tilde{\x}_* - \sum_{l:\theta_l = \pi } x_l \e_l \Big) , \quad
\c_1 = U^\dag \Big(\tilde{\x}_* + \sum_{l:\theta_l = \pi } x_l \e_l \Big) .
\]
This implies $\c_0+\c_1 = 2\x_*$.
\end{proof}

As a corollary, we can compute the center of the minimal sphere that supports $\S_{kj}$ for any $k,j$. More precisely, we have the following result.

\begin{cor}
\label{cor1}
Assume that the unit eigenvectors of $\R_A$ corresponding to $-1$ are $\u_1,\ldots,\u_q$, then the center of the minimal sphere supporting $\S_{kj}$ is
$
\c_{kj} = (I_n -\R_k \cdots \R_1) \x_*  + \R_k \cdots \R_1 \c_j,
$\footnote{When $k=0$, $\R_k \cdots \R_1$ is understood as the identity matrix.}
where
\[
\c_j = \x_0 + \Big(I_n - (-1)^j \sum_{l=1}^q \u_l \u_l^{\TT} \Big) (\x_* - \x_0), \quad j\in\{0,1\}.
\]
Consequently, $\c_{k0}+\c_{k1}=2\x_*$ for all $k=0,1,\ldots,n-1$.
\end{cor}

\begin{proof}
As we can see from the proof of Theorem \ref{main thm}, when $\x_0 = 0$ the centers of $\S_0, \S_1$ are respectively
\[
\c_0 = \Big(I_n  - \sum_{l=1}^q \u_l \u_l^{\TT}\Big) \x_* , \quad
\c_1 = \Big(I_n  + \sum_{l=1}^q \u_l \u_l^{\TT}\Big) \x_* .
\]
To understand the center of $\S_{kj}$, we first respectively compute the centers of the minimal spheres supporting $\S_0,\S_1$ for $\x_0 \neq 0$. When $\x_0\neq 0$, we have
\[
\y_i = \x_* + \R_A^{2i}(\x_0-\x_*) = \x_0 + (I_n -\R_A^{2i}) (\x_* - \x_0),
\]
so we can focus on $\{\y_i - \x_0:i=0,1,2,\ldots\} = \S_0 - \x_0$ in the proof of Theorem \ref{main thm}, which is a translation of $\S_0$. Similar to the proof of Theorem \ref{main thm}, we have
\[
\c_0 = \x_0 + \Big(I_n  - \sum_{l=1}^q \u_l \u_l^{\TT}\Big) (\x_* - \x_0) , \quad
\c_1 = \x_0 + \Big(I_n  + \sum_{l=1}^q \u_l \u_l^{\TT}\Big) (\x_* - \x_0) .
\]

We next compute the center of the minimal sphere supporting $\S_{kj}$.
From the above analysis and note that the vectors in $\S_{kj}$ have the form $\x_k + (I_n - (\R_k\cdots \R_1 \R_n\cdots \R_{k+1})^{2i+j}) (\x_*-\x_k)$ by (\ref{sphere 3}), we have that the center of the minimal sphere supporting $\S_{kj}$ equals
\bes
\c_{kj} =  \x_k + \Big(I_n - (-1)^j \sum_{l=1}^q \u_{kl} \u_{kl}^{\TT} \Big) (\x_* - \x_k),
\ees
where $\u_{kl}$ are the unit eigenvectors of $\R_k\cdots \R_1 \R_n\cdots \R_{k+1}$ corresponding to the eigenvalue $-1$.
Since 
\[
\R_k\cdots \R_1 \R_n\cdots \R_{k+1}
= (\R_k\cdots \R_1) \R_A (\R_k\cdots \R_1)^{\TT},
\]
it follows that $\u_{kl} = \R_k\cdots \R_1 \u_l$. 
Also note that $\x_k = (I_n -\R_k \cdots \R_1) \x_* + \R_k \cdots \R_1 \x_0$, so we can rewrite the center as
\beas
\c_{kj} &=& (I_n -\R_k \cdots \R_1) \x_* + \R_k \cdots \R_1 \x_0  + \R_k \cdots \R_1  \Big(I_n  - (-1)^j \sum_{l=1}^q  \u_{l} \u_{l}^{\TT}  \Big) (\x_* - \x_0) \\
&=& (I_n -\R_k \cdots \R_1) \x_*  + \R_k \cdots \R_1 \c_j.
\eeas
This completes the proof.
\end{proof}

From Corollary \ref{cor1}, we see that if $n$ is odd, then $q \geq 1$ and so $\c_0 \neq \c_1 \neq \x_*$ unless $\x_*$ has no overlap in the eigenspace of $\R_A$ corresponding to the eigenvalue $-1$. However, if $n$ is even, it may happen that $q = 0$. In this case $\c_0 = \c_1 = \x_*$. When this happens, $\S_0, \S_1$ are on the same minimal sphere.

\cs{Theorem \ref{main thm} implies the correctness of Algorithm \ref{Kaczmarz algorithm}.} Next, we consider its efficiency.

\begin{thm}
\label{main thm2}
\cs{Let $A$ be an $n\times n$ invertible matrix, $\x_0, \b$ be two $n\times 1$ vectors, and $\x_* = A^{-1}\b$. Let $\y_j,\z_j$ be given in (\ref{sphere 0}), (\ref{sphere 1}), then}
\be
\label{main thm2:eq1}
\|\frac{1}{2N} \sum_{j=0}^{N-1} (\y_j+\z_j) - \x_*\| \leq \varepsilon \|\x_0-\x_*\|,
\ee
where 
\be
N = \left\lceil \frac{\pi \eta(A)}{2\varepsilon } \right\rceil
\ee
and $\eta(A)$ is given in Definition \ref{defn:eta}.
\end{thm}

\begin{proof}
Denote
\bes
\tilde{\x} = \frac{1}{2N} \sum_{j=0}^{N-1} (\y_j+\z_j).
\ees
From the \cs{construction of $\y_j,\z_j$}, we have
\bes
\tilde{\x} = \frac{1}{2N} \sum_{j=0}^{2N-1} \Big(\x_* + \R_A^j (\x_0-\x_*) \Big)
=\x_* + \frac{1}{2N} \sum_{j=0}^{2N-1} \R_A^j (\x_0-\x_*).
\ees
Assume that the unit eigenvectors of $\R_A$ are $\u_l, l \in \{1,\ldots,n\}$. We decompose $\x_0 - \x_* = \sum_{l=1}^n \beta_l \u_l$, \cs{then by the orthogonality of $\{\u_l\}_{l=1}^n$, it is easy to show that}
\begin{eqnarray*}
\|\tilde{\x} - \x_*\|^2 &=& \frac{1}{4N^2} \left\| \sum_{j=0}^{2N-1} \R_A^j (\x_0 - \x_*) \right\|^2 \\
&=& \frac{1}{4N^2} \left\|  \sum_{l=1}^n  \beta_l \left(\sum_{j=0}^{2N-1} e^{ij\theta_l}\right)  \u_l \right\|^2 \\
&=& \frac{1}{4N^2}  \sum_{l=1}^n  |\beta_l|^2 \left|\sum_{j=0}^{2N-1} e^{ij\theta_l} \right|^2 \\
&=& \frac{1}{4N^2}  \sum_{l=1}^n  |\beta_l|^2
\left|\frac{1-e^{2iN\theta_l}}{1-e^{i\theta_l}}\right|^2 \\
&=& \frac{1}{4N^2}  \sum_{l=1}^n  |\beta_l|^2
\frac{\sin^2(N\theta_l)}{\sin^2(\theta_l/2)}.
\end{eqnarray*}

Notice that if $0 < |\theta_l| \leq \pi$, we have $|\sin(\theta_l/2)| \geq |\theta_l|/\pi$. This means that ${ \sin^2(N\theta_l)}/{ \sin^2(\theta_l/2) } \leq \pi^2/\theta_l^2 $. By Lemma \ref{key lemma 2}, $\theta_l \neq 0$ for all $l$. Hence,
\[
\|\tilde{\x} - \x_*\|^2
\leq \frac{\pi^2\|\x_0 - \x_*\|^2}{4N^2\min_l|\theta_l|^2} .
\]
Set the above estimate as $\varepsilon^2\|\x_0 - \x_*\|^2$, we then have
\bes
N = \left\lceil \frac{\pi}{2\varepsilon \min_l|\theta_l|} \right\rceil = \left\lceil \frac{\pi \eta(A)}{2\varepsilon } \right\rceil.
\ees
This completes the proof.
\end{proof}

\cs{As already mentioned in the introduction, based on the above result, to obtain an approximation of the solution up to relative error $\varepsilon$, we can modify Algorithm \ref{Kaczmarz algorithm} by only using $O(\eta(A)\log(1/\varepsilon))$ iterations. The basic idea is as follows: In Equation (\ref{main thm2:eq1}), we choose $\varepsilon=1/2$, then $N = \lceil \pi \eta(A)\rceil$ and $\|\tilde{\x}_0-\x_*\| \leq 0.5 \|\x_0-\x_*\|,$ where 
$\tilde{\x}_0 = \frac{1}{2N} \sum_{j=0}^{N-1} (\y_j+\z_j)$. We now view $\tilde{\x}_0$ as the new initial vector and repeat the above procedure. After $\log(1/\varepsilon)$ iterations, we obtain an approximation of the solution up to relative error $\varepsilon$. The total number of iterations is $O(\eta(A)\log(1/\varepsilon))$. 

}

In Algorithm \ref{Kaczmarz algorithm}, we only focused on $\x_{jm}$ for $j\in\{0,1,\ldots,M-1\}$. So it costs $O(mn)$ to compute $\x_{(j+1)m}$  from $\x_{jm}$. Previous Kaczmarz algorithms usually use $O(n)$ operations at each step of the iteration. We can solve this problem by proposing the following algorithm, the correctness of which is guaranteed by Corollary \ref{cor1} and Theorem \ref{main thm2}.

\begin{algorithm}[H]
\caption{\cs{The second deterministic iterative reflection algorithm}}
\label{Kaczmarz algorithm2}
\begin{algorithmic}[1]
\Require A consistent linear system $A\x=\b$, an arbitrarily chosen initial vector $\x_0$, and $\varepsilon\in[0,1]$.

\Ensure \cs{Output $\tilde{\x}$ such that $\|A\tilde{\x} - \b\|\leq \varepsilon$.}


\State For $k=0,1,2,\cdots $, compute
\be \label{our procedure}
\x_{k+1} = \x_k + 2 \frac{b_{i_k} -A_{i_k}^{\TT}\x_k}{\|A_{i_k}\|^2}   A_{i_k},
\ee
where $i_k= (k\mod m) + 1$.


\State \cs{Stop the iteration at step $N$ if
$\|A\tilde{\x} - \b\| \leq \varepsilon$, where $\tilde{\x} = \frac{1}{N} \sum_{j=0}^{N-1} \x_{j}$.}

\State \cs{Output $\tilde{\x}$.}
\end{algorithmic}
\end{algorithm}

Compared with Algorithm \ref{Kaczmarz algorithm}, this algorithm is easier to implement because it only uses $O(n)$ operations at each step of the iteration. But it requires more steps to converge. The overall complexity of these two algorithms is the same in theory. 
To run Algorithm \ref{Kaczmarz algorithm2} in practice, we can use a similar idea introduced above to reduce the dependence on $\varepsilon$ to $\log(1/\varepsilon)$. That is we compute some $\tilde{\x}$ and restart the iteration from  $\tilde{\x}$ if it is not a good approximation of the solution. It turns out that this is very efficient in practice, see the numerical results in Section \ref{section:Numerical experiments}.


The convergence rates of the algorithms proposed in this paper highly depend on $\eta(A)$. To obtain a better understanding, at the end of this section, we build its connection to the condition number of $A$. Theoretically, we have the following general upper bound.

\begin{prop}
\label{prop:upper bound of convergence rate}
\cs{Let $A$ be an $m\times n$ matrix, then $\eta(A) = O((\log m) \kappa^2(A))$.}
\end{prop}

\begin{proof}
By Lemma \ref{lem:Brady-Watt formula}, we have $\R_A = I_n - 2A^{\TT} W^{-1} A$, where $W$ is the lower-triangular matrix such that $W+W^{\TT} = 2AA^{\TT}$. Now assume that $e^{i\theta}$ is the eigenvalue \cs{of $\R_A$} such that $\theta$ is minimal and nonzero. Set the corresponding unit eigenvector as $\u$, then we have
$|1-e^{i\theta}| = 2\|A^{\TT} W^{-1} A\u\|$.
We consider the case that $\theta>0$ is close to 0 so that the left-hand side is as small as $\theta$. The right-hand side is lower bounded by $2\sigma_{\min}(A)^2 \sigma_{\min}(W^{-1})$, where $\sigma_{\min}$ refers to the minimal nonzero singular value. Thus $ 1/\theta = O(\|W\|/\sigma_{\min}(A)^2)$. Since $W+W^{\TT} = 2AA^{\TT}$, we know that $W = (2AA^{\TT})\circ T - {\rm diag}(\|A_1^{\TT}\|,\cdots,\|A_m^{\TT}\|)$, where $T$ is the triangular truncation operator defined above Lemma \ref{lem: triangular truncation operator}.
By Lemma \ref{lem: triangular truncation operator}, we have
\beas
\|W\| \leq \frac{\log m}{\pi} \|2AA^{\TT}\| + \max_{1\leq i \leq m} \|A_i^{\TT}\|
\leq \frac{2\log m}{\pi} \|A\|^2 +  \|A\|.
\eeas
Therefore, we have
$1/\theta = O((\log m) \kappa^2(A))$.
\end{proof}

The upper bound in the above proposition can \cs{be reached} in some cases. For example, consider the case $m=n=3$. 
For simplicity, we assume that
\[
A = \begin{pmatrix}
1 & 0 & 0 \\
\cos(x) & \sin(x) & 0 \\
\cos^2(x) & \cos(x) \sin(x) & \sin(x)  \\
\end{pmatrix},
\]
where $\sin(x)\neq 0$. Then direct calculation shows that the singular values of $A$ are
\[
|\sin(x)|, \quad \sqrt{\frac{1}{2} \left( 2 + \cos^2(x) \pm |\cos(x)| \sqrt{8+\cos^2(x)} \right)}.
\]
Let $x$ tend to 0, then the singular values are approximately equal to
$
|x|, 3, |x|/3.
$
Thus the condition number is $\kappa(A) = \Theta(1/|x|)$. We can also compute that the eigenvalues of $\R_A$ are 
\[
-1, \,\,\,
1-2\sin^4(x) \pm 2i \sin^2(x) \sqrt{1-\sin^4(x) }.
\]
When $x$ is close to 0, the eigenvalues are close to $-1, 1\pm 2i x^2$, which means $\min |\theta_l| \approx x^2 \approx 1/\kappa^2(A)$. Note that  when $x\approx 0$, the rows of $A$ are close to each other and $\R_A$ is close to ${\rm diag}\{-1,1,1\}$.  In the high-dimensional case, as suggested by the numerical tests, for matrices with the same structure as $A$ (i.e., the $i$-th row is the polar coordinate of a unit vector, all rows use the same parameter $x$), we always have $\min |\theta_l|  \approx 1/\kappa^2$ when $x\approx 0$. 

{
Below we show that for random matrices, $\eta(A)\approx \kappa(A)$. Let $U$ be a random orthogonal matrix taken from the orthogonal group of dimension $n$ uniformly according to the Haar measure. Denote the eigenvalues of $U$ by $\{e^{i\theta_1},\ldots, e^{i\theta_n}\}$. For any $\theta\geq 0$, denote $\mathcal{N}_{\theta}=\#\{j:0\leq \theta_j \leq \theta\}$. Then according to random matrix theory \cite[Proposition 4.7]{meckes2019random}, we have
\be
\label{eq for eta}
\left|\E[\mathcal{N}_{\theta}] - \frac{n\theta}{\pi} \right| \leq \frac{1+\pi}{2\pi}.
\ee
If we choose $\theta = 1/10n$, then $\E[\mathcal{N}_{\theta}] \leq (1+\pi)/2\pi + 1/10\pi < 0.7$. By Markov's inequality
\[
{\rm Prob}[\mathcal{N}_{\theta} > 1] \leq \E[\mathcal{N}_{\theta}] < 0.7.
\]
Thus  $\mathcal{N}_{\theta} = 0$ with probability at least 0.3. Equivalently, with probability at least 0.3, we have $\min|\theta|\geq 1/10n$.


If we pick a point $\v$ uniformly random on the unit sphere of dimension $n$, then $I_n - 2\v\v^T$ defines a reflection in the hyperplane orthogonal to $\v$. We obtain a random orthogonal matrix by forming a product of independent copies of these reflections. The question is how many random reflections are required to get close to the Haar measure. In \cite[page 187]{matrices1986random}, it was shown that $0.5n\log n + \alpha n$ random reflections are required, where $\alpha$ is a universal constant independent of $n$. So $\R_A$ is a random orthogonal matrix under the Haar measure when $m=\Omega(n\log n)$.
In this case, from the above analysis, $\eta(A) < 10n$ with probability at least 0.3, and $\eta(A) \approx n/\pi \approx 0.32n$ on average from (\ref{eq for eta}). In summary, we have the following result.

\begin{prop}
\label{prop:relation of eta and kappa}
Let $A$ be an $m \times n$ random matrix whose rows are uniformly random vectors. Assume that $m=\Omega(n\log n)$. Then $\eta(A)\approx n/\pi$ on average and $\eta(A) < 10 n$ with probability at least 0.3.
\end{prop}

Assume that $A$ is $n\times n$ and that all the entries independently follow the standard normal distribution. Then it was shown in \cite[Theorem 7.1]{edelman1988eigenvalues} that
$
\E[\log \kappa(A)] = \log n + c + o(1)
$
when $n$ is sufficiently large, where $c\approx 1.537$. So on average, $\kappa(A) \approx e^{1.537}n\approx 4.65n$. It follows from Corollary 7.1 of \cite{edelman1988eigenvalues} that 
$
{\rm Prob}[\kappa(A) < x n] = e^{-2/x-2/x^2}.
$
For example, if we choose $x=10$, then $\kappa(A) < 10 n$ with probability close to 0.8.\footnote{As shown in \cite[Corollary 3.3]{tao2010smooth}, if the entries of a matrix take values iid from a distribution with mean zero, then the condition number is smaller than $2n$ with high probability.}
Therefore, for random Gaussian matrices, on average 
$\eta(A) \approx \kappa(A) \approx n.$
This result is better than Proposition \ref{prop:upper bound of convergence rate}.

}



\subsection{Case 2: $A$ has full row-rank}

We now turn to the general case. Assume that $A$ is $m\times n$. We first focus on the case that $A$ has full row-rank $m$. Then we extend our ideas to general consistent linear systems. 

For consistent linear systems with $m<n$, there are an infinite number of solutions. Although the situation looks complicated now, we will see that some interesting properties will appear. The definitions of $\R_i, \R_A$ are the same as those in (\ref{reflection}), (\ref{prod of reflections}). As for $ \S_0,\S_1$, the change is slight. Compared to (\ref{sphere 0}), (\ref{sphere 1}), they now become
$
\S_0 =
\{\x_0, \x_{2m}, \x_{4m}, \ldots \}, 
\S_1 =
\{\x_{m}, \x_{3m}, \x_{5m}, \ldots \}.
$
Below, $\{\x_0, \x_1, \ldots\}$ still refers to the series of vectors generated by the procedure (\ref{our procedure}).

The following result holds for all consistent linear systems.

\begin{thm}
\label{thm1:singular case}
Assume that $A\x = \b$ is consistent.
Let $\x_0$ be an arbitrarily chosen initial vector, then $\{\x_0,\x_1,\x_2,\ldots\}$ \cs{generated by the process (\ref{our procedure})} lies on a sphere 
centered on the solution that has the minimal distance to $\x_0$.
\end{thm}

\begin{proof}
We use $\x_*$ to denote the solution of $A\x = \b$ that has the minimal distance to $\x_0$.
For any solution $\x$ of the linear system, we have $b_i = \langle A_i|\x\rangle$. Thus
$\x_{k+1} - \x = \R_i (\x_k - \x)$, where $i = (k \mod m) + 1$. This means $\|\x_{k+1} - \x\| = \|\x_k - \x\|$. Moreover, $\langle \x_{k+1} - \x_*| \x - \x_*\rangle = \langle \x_k - \x_*|\R_i| \x - \x_*\rangle = \langle \x_k - \x_*|\x - \x_*\rangle$. This implies that if $\langle \x_0 - \x_*|\x - \x_*\rangle = 0$, then $\langle \x_k - \x_*|\x - \x_*\rangle = 0$ for all $k$. Equivalently, if $\x_*$ is the solution that has the minimal distance to $\x_0$, then $\x_*$ is the solution that has the minimal distance to all $\x_k$.
\end{proof}

Figure \ref{circles in the singular case} in the introduction is an illustration of Theorem \ref{thm1:singular case} in dimension 3. As a result of Theorem \ref{thm1:singular case}, we can approximate any solution of $A\x = \b$ that has the minimal distance to a given vector we \cs{are interested in} using Algorithm \ref{Kaczmarz algorithm}. For example, if we start from $\x_0=0$, we can approximate the solution that has the minimal norm. 

Theorems \ref{main thm}, \ref{main thm2} are not influenced too much. 
We first prove an analogy of Theorem \ref{main thm2}.  

\begin{thm}
\label{thm2:singular case}
Assume that $A$ is $m\times n$ and ${\rm Rank}(A)=m$, then
\be
\|\frac{1}{N} \sum_{i=0}^{N-1} \x_{im} - \x_*\| \leq \varepsilon \|\x_0-\x_*\|,
\ee
where \cs{$N=O(\eta(A)/\varepsilon)$},
and $\x_*$ is the solution that has the minimal distance to $\x_0$.
\end{thm}

\begin{proof}
The proof here is similar to that of Theorem \ref{main thm2}. It suffices to show that $\x_*$ has no overlap in the eigenspace of $\R_A$ corresponding to the eigenvalue 1. With this result, the estimation of $N$ then follows directly by a similar argument to the proof of Theorem  \ref{main thm2}.

By Brady-Watt's formula (see Lemma \ref{lem:Brady-Watt formula}) $\R_A = I_n - 2A^{\TT} W^{-1} A$ and the fact that $A$ has full row-rank, we know that $1$ is an eigenvalue of $\R_A$ with multiplicity $n-m$.
Moreover, assume that $\u$ is an eigenvalue of $\R_A$ corresponding to 1, i.e., $\R_A \u = \u$, then $A\u = 0$. Indeed, from Brady-Watt's formula, we have $A^{\TT} W^{-1} A \u = 0$. Since $A$ has full row-rank, $W^{-1} A \u = 0$ which further implies $A\u=0$ because of the non-singularity of $W$.
Since $\x_*$ is the solution that has the minimal distance to $\x_0$, we have that $\x_0-\x_*$ is a linear combination of eigenvectors of $\R_A$ that are not corresponding to the eigenvalue 1. Indeed, if there is an overlap, say $\v$, on the eigenspace of $\R_A$ corresponding to the eigenvalue 1, then the distance between $\x_0$ and $\x_*+\v$ is strictly smaller than the distance between $\x_0$ and $\x_*$. This is a contradiction in that $\x_*+\v$ is also a solution  of $A \x = \b$.
\end{proof}

\begin{thm}
\label{thm11:singular case}
Assume that ${\rm Rank}(A)=m$.
For $j\in\{0,1\}$,  let $\c_j$ be the center of the minimal sphere supporting $\S_j(\x_0)$, then $\c_0+\c_1 = 2\x_*$, where $\x_*$ is the solution that has the minimal distance to $\x_0$.
\end{thm}

\begin{proof}
Based on the proof of Theorem \ref{main thm}, the argument here is greatly simplified. Note that here we need to focus on $\y_j - \x_0, \z_j - \x_0$ and $ \x_* - \x_0 $. The notation below is the same as that in the proof of Theorem \ref{main thm}. Denote $\tilde{\x}_* = U(\x_* - \x_0) = \sum_{k=1}^n x_k \e_k$, then from the proof of Theorem \ref{main thm} we know that
\beas
\c_0 &=& \x_0 + U^{\dag} \left(\tilde{\x}_* - \sum_{l:\theta_l = \pi } x_l \e_l - \sum_{l:\theta_l = 0 } x_l \e_l \right) 
= \x_* - U^{\dag} \left(\sum_{l:\theta_l = \pi } x_l \e_l + \sum_{l:\theta_l = 0 } x_l \e_l\right), \\
\c_1 &=& \x_0 + U^{\dag} \left(\tilde{\x}_* + \sum_{l:\theta_l = \pi } x_l \e_l - \sum_{l:\theta_l = 0 } x_l \e_l \right) 
= \x_* + U^{\dag} \left(\sum_{l:\theta_l = \pi } x_l \e_l - \sum_{l:\theta_l = 0 } x_l \e_l\right).
\eeas
From the proof of Theorem \ref{thm2:singular case}, we know that $\x_* - \x_0 $ has no overlap in the eigenspace of $\R_A$ corresponding to the eigenvalue 1. This means that the summations \cs{over $l$ with $\theta_l=0$} on the right hand side of $\c_0,\c_1$ are zero.
Therefore, we have $\c_0 + \c_1 = 2\x_*$.
\end{proof}

\subsection{Case 3: $A$ does not have full row-rank}
\label{Case 3: A does not have full row-rank}

In this section, we consider general consistent linear systems. We will show that Theorems \ref{thm2:singular case}, \ref{thm11:singular case} are almost correct. 
From the proof of Theorem \ref{thm2:singular case}, we know that a key point to the success of \cs{Algorithm \ref{Kaczmarz algorithm}} is that $\x_0-\x_*$ has no overlap in the eigenspace of $\R_A$ corresponding to the eigenvalue $1$. This result is usually no longer correct when $A$ does not have full row rank. Below we give a simple method to fix this. We will show that if we randomly introduce some new linear constraints that do not change the solution set, then the new linear system will satisfy the above-expected condition with high probability. This idea may fail, while we will show that the set of the new linear constraints that makes the algorithm fails is a set with Lebesgue measure zero.

\begin{lem}
\label{lem:determinant}
Let $A$ be a $r\times r$ invertible matrix, $C$ be a $l\times r$ matrix with no zero rows. Let $B$ be the $l\times l$ lower triangular matrix such that $B+B^{\TT} = C(A+A^{\TT})C^{\TT}$. If $B$ is invertible, then
\[
\det(I_r - A^{\TT}C^{\TT}B^{-1}C)
=
\det(I_l - CA^{\TT}C^{\TT}B^{-1}).
\]
\end{lem}

\begin{proof}
For convenience, we denote $M = I_r - A^{\TT}C^{\TT}B^{-1}C$ and $N = I_l - CA^{\TT}C^{\TT}B^{-1}$. Denote the $i$-th row of $C$ as $\c_i^{\TT}$.

We first suppose that ${\rm Rank}(C) = l$.  Let the singular value decomposition of $C$ be  $C = UDV^{\TT}$, where $D$ is a matrix of the form $ (\Lambda_{l\times l}, \,\, 0_{l\times (r-l)} )$ and $\Lambda_{l\times l}$ is an invertible diagonal matrix. It is easy to check that $C M = NC$. So we have $D V^{\TT} M V = U^{\TT} N U D$. We decompose $V^{\TT} M V = \begin{pmatrix}
M_1 & M_2 \\
M_3 & M_4
\end{pmatrix}$, where $M_1$ is $l\times l$, then  
\bes
\Big(\Lambda_{l\times l}, \,\, 0_{l\times (r-l)} \Big)
\begin{pmatrix}
M_1 & M_2 \\
M_3 & M_4
\end{pmatrix}
= U^{\TT}N U \Big(\Lambda_{l\times l}, \,\, 0_{l\times (r-l)} \Big).
\ees
So $\Lambda_{l\times l} M_1 = U^{\TT}N U \Lambda_{l\times l}$ and $M_2 = 0$. 
Note that $M$ is invariant in the orthogonal complementary space generated of the rows of $C$. Namely, for any column vector $\q\in \mathbb{R}^r$ with $\c_1^{\TT} \q =\cdots=\c_l^{\TT} \q = 0$, we have $M\q=\q$. This means $V^{\TT} M V (V^{\TT}\q) =V^{\TT}\q$. So $M_4 = I_{r-l}$. Thus $\det(M)  = \det(M_1)= \det(N)$. 

Now assume that ${\rm Rank}(C)<l$. We still use the above notation but with slight changes in the meaning. The whole analysis is similar. 
Now we decompose
\[
D = 
\begin{pmatrix}
\widetilde{\Lambda}_{t\times t} & 0 \\
0 & 0
\end{pmatrix},  \quad 
V^{\TT} M V = \begin{pmatrix}
M_{1} & M_{2} \\
M_{3} & M_{4}
\end{pmatrix},  \quad
U^{\TT}N U = \begin{pmatrix}
N_{1} & N_{2} \\
N_{3} & N_{4}
\end{pmatrix}, 
\]
where $M_1,N_1$ are $t\times t$ matrices. By $D V^{\TT} M V = U^{\TT} N U D$, we have
$
\widetilde{\Lambda}_{t\times t} M_1 = N_1 \widetilde{\Lambda}_{t\times t}, M_2 = 0, N_3 = 0.
$
Note that for any left singular vector $\u$ corresponding to the singular value 0 of $C$, we have $\u^{\TT}N = \u^{\TT}$. For any right singular vector $\v$ corresponding to the singular value 0 of $C$, we have $M\v=\v$. Thus $M_4  = I_{r-t},  N_4 = I_{l-t}$. Consequently, $\det(M) = \det(M_1) = \det(N_1) = \det(N)$.
\end{proof}


In the above lemma, note that $NB = B - CA^{\TT}C^{\TT}$ is skew-symmetric, so $\det(I_r - A^{\TT}C^{\TT}B^{-1}C)
=
\det(I_l - CA^{\TT}C^{\TT}B^{-1}) = 0$ if $l$ is odd. The above proof shows that without counting the multiplicity, $M, N$ have the same \cs{set of} eigenvalues. Moreover, when counting the multiplicity, the only difference is the multiplicity of \cs{the} eigenvalue 1. The difference is $|r-l|$.

For any $r$ vectors $\{\v_1,\ldots,\v_r\} \subset \mathbb{R}^n$, we use ${\rm Span}_{\mathbb{R}}\{\v_1,\ldots,\v_r\}$ to denote the $\mathbb{R}$-linear space spanned by $\{\v_1,\ldots,\v_r\}$.

\begin{prop}
\label{prop: eigenvalue 1 of products of reflections}
Let $\{\v_1,\ldots,\v_r\}$ be a set of linearly independent column vectors of $\mathbb{R}^n$, $\u_1,\ldots,\u_l\in {\rm Span}_{\mathbb{R}}\{\v_1,\ldots,\v_r\}$. Denote $V$ as the $r\times n$ matrix with $i$-th row equals $\v_i^{\TT}$. Also, denote
\[
\R_v = \prod_{i=1}^r (I_n - 2 \frac{\v_i \v_i^{\TT}}{\|\v_i\|^2}), \quad 
\R_u =  \R_v \prod_{i=1}^l (I_n - 2 \frac{\u_i \u_i^{\TT}}{\|\u_i\|^2}).
\]

\begin{enumerate}
\item[(a)] If $\R_v\w = \w$, then  $V\w = 0, \R_u \w = \w$. Conversely, if $V\w = 0$, then $\R_v\w = \w$.
\item[(b)] Let $S_{v}$ be the multiplicity of \cs{the} eigenvalue 1 of $\R_v$, and $S_{u}$ be the multiplicity of \cs{the} eigenvalue 1 of $\R_u$.
Let $P$ be the $l\times l$ skew-symmetric matrix defined by
\be \label{skew-symmetric matrix}
P(i,j) = \c_i^{\TT} W \c_j \text{~for~}i>j,
\ee
where $\u_i = V^{\TT} \c_i$, and $W$ is the lower-triangular matrix such that $W+W^{\TT} = 2VV^{\TT}$. If $l$ is odd, then $S_{u} \geq S_{v} + 1$. If $l$ is even and $\det(P) \neq 0$, then $S_{u} = S_{v}$.
\end{enumerate}

\end{prop}

\begin{proof}
(a). 
By Brady-Watt's formula (see Lemma \ref{lem:Brady-Watt formula}), we have $\R_v = I_n - 2V^{\TT} W^{-1} V$ for some invertible lower triangular matrix $W$ such that $W+W^{\TT} = 2VV^{\TT}$. If $\R_v \w = \w$, we have $V^{\TT} W^{-1} V \w = 0$. Since $V$ has full row-rank, $V\w=0$. The converse statement is obvious. Since $\u_i$ is a linear combination of $\{\v_1,\ldots,\v_r\}$, we know that $\u_i^{\TT} \w = 0$. Thus, $\R_u \w = \w$.

(b). From the proof of (a), we know that $S_v \leq S_u$. Moreover, when $l$ is odd we have $\det(\R_u) = - \det(\R_v)$. Since complex eigenvalues of $\R_u, \R_v$ appear in conjugate forms, we know that $S_v<S_u$ if $l$ is odd. Next, we describe the condition \cs{that ensures} $S_u = S_v$ when $l$ is even.

For convenience, denote
\[
C = \begin{pmatrix}
\c_1^{\TT} \\
\vdots \\
\c_l^{\TT} \\
\end{pmatrix}_{l\times r}, \quad
U = \begin{pmatrix}
\u_1^{\TT} \\
\vdots \\
\u_l^{\TT} \\
\end{pmatrix}_{l\times n} = CV,
\]
Let $W_1$ be the lower-triangular matrix so that $W_1+W_1^{\TT} = 2 UU^{\TT} = 2 CVV^{\TT} C^{\TT} = C (W+W^{\TT}) C^{\TT}$. This means
\[
W_1(i,j) = \begin{cases}
\c_i^{\TT} W \c_i & i=j, \\
\c_i^{\TT} (W+W^{\TT}) \c_j & i>j. \\
\end{cases}
\]

Now assume that $\R_u \x = \x$, i.e., $\R_v^{\TT} \x = \prod_{i=1}^l (I_n - 2 \frac{\u_i \u_i^{\TT}}{\|\u_i\|^2}) \x$. Then by Brady-Watt's formula, we have
\[
(I_n - 2 V^{\TT} W^{-T} V) \x 
= (I_n - 2 U^{\TT} W_1^{-T} U) \x =
(I_n - 2 V^{\TT} C^{\TT} W_1^{-T} C V) \x .
\]
That is $V^{\TT} W^{-T} (I_r - W^{\TT} C^{\TT} W_1^{-T} C) V \x = 0$. Denote $M=I_r - W^{\TT} C^{\TT} W_1^{-T} C$, which is a $r\times r$ matrix. Since $V$ has full-column rank, $V\x=0$ if $M$ is invertible. From (a), $V\x=0$ if and only if $\R_v \x = \x$. So we need to figure out the condition that can make sure that $M$ is invertible.
By Lemma \ref{lem:determinant}, $\det(M) = \det(W_1 - C W^{\TT} C^{\TT})\det(W_1^{-T})$. From the construction, we know that $W_1 - C W^{\TT} C^{\TT}$ is a skew-symmetric matrix with  $(i,j)$-th entry equals $\c_i^{\TT} W \c_j$ for $i>j$. It is the matrix $P$ defined in (\ref{skew-symmetric matrix}). When $l$ is odd, $\det(P) = 0$, which further proves that $S_u > S_v$. When $l$ is even, $S_u = S_v$ if $\det(P) \neq 0$.
\end{proof}

Let $P$ be the skew-symmetric matrix defined in Proposition \ref{prop: eigenvalue 1 of products of reflections}, $\det(P) = 0$ defines a hypersurface with Lebesgue measure 0. So if we randomly choose $\c_1, \ldots,\c_l$, we usually have $\det(P) \neq 0$. In practice, we can just choose $l=2$. The following corollary follows directly from Proposition \ref{prop: eigenvalue 1 of products of reflections}. It states that when the vectors $\{\v_1,\ldots,\v_r\}$ are linearly dependent, we can still find extra vectors such that the eigenspace corresponding to the eigenvalue 1 does not change.

\begin{cor}
\label{cor:a useful result}
Let $\{\v_1,\ldots,\v_p\}$ be linearly independent column vectors in $\mathbb{R}^n$, $\v_{p+1},\ldots,\v_{r}$ be any $(r-p)$ column vectors from ${\rm Span}_{\mathbb{R}}\{\v_1,\ldots,\v_p\}$. Then there exist an integer $l$ and $\u_1,\ldots,\u_l\in {\rm Span}_{\mathbb{R}}\{\v_1,\ldots,\v_r\}$ such that 
\[
\R_v = \prod_{i=1}^{p} (I_n - 2 \frac{\v_i \v_i^{\TT}}{\|\v_i\|^2}), \quad 
\R_u = \R_v \prod_{i=p+1}^{r} (I_n - 2 \frac{\v_i \v_i^{\TT}}{\|\v_i\|^2}) \prod_{j=1}^l (I_n - 2 \frac{\u_j \u_j^{\TT}}{\|\u_j\|^2})
\]
have the same eigenspace corresponding to the eigenvalue 1.
\end{cor}

\begin{proof}
From Proposition \ref{prop: eigenvalue 1 of products of reflections}, we need to choose $l$ so that $l+r-p$ is even. We also need to make sure that the skew-symmetric matrix defined by (\ref{skew-symmetric matrix}) is nonsingular. Since $\det(P) = 0$ defines a set with zero Lebesgue measure, such $\u_1,\ldots,\u_l$ always exist.
\end{proof}

Making the same assumption as the above corollary, the eigenspace corresponding to the eigenvalue 1 of the operator $\R_v \prod_{i=p+1}^{r} (I - 2 \frac{\v_i \v_i^{\TT}}{\|\v_i\|^2})$ may not be equal to that of $\R_v$. However, we can introduce some extra vectors to modify it. To this end, we introduce the following concept.

\begin{defn}[Reflection consistent]
\label{defn:Reflection consistent}
Let $A$ be an $m\times n$ matrix of rank $r$. Let $\a_1^{\TT},\ldots,\a_r^{\TT}$ be $r$ linearly independent rows of $A$. If $\R_A$ and $ \prod_{i=1}^{r} (I_n - 2 \frac{\a_i \a_i^{\TT}}{\|\a_i\|^2})$ have the same eigenspace corresponding to \cs{the} eigenvalue 1, then $A$ is called reflection consistent.
\end{defn}

From (a) of Proposition \ref{prop: eigenvalue 1 of products of reflections}, the above concept is independent of the choice of $\a_1^{\TT},\ldots,\a_r^{\TT}$. 
By Proposition \ref{prop: eigenvalue 1 of products of reflections}, \cs{we have the following necessary and sufficient condition for reflection consistency.}

\cs{
\begin{cor}
\label{cor:reflection consistence}
Let $A$ be an $m\times n$ matrix of rank $r$. Let $B$ be any matrix consisting of $r$ linearly independent rows of $A$. For any row $\a_j^{\TT}$ of $A$ not in $B$, assume that $\a_j^{\TT} = \c_j^{\TT} B$. Let $P$ be the $(m-r)\times(m-r)$ skew-symmetric matrix with $(i,j)$-th entry defined by
\[
P(i,j) = \c_i^{\TT} W \c_j~\text{for}~i>j,
\]
where $W$ is the lower-triangular matrix such that $W+W^{\TT} = 2BB^{\TT}$. Then $A$ is reflection consistent if and only if $\det(P)\neq 0$ and $m-r$ is even.
\end{cor}

Certainly, $\det(P)\neq 0$ implies that $m-r$ is even. The latter condition is usually easy to check especially when $A$ has full column rank, so we keep it in the statement of the above corollary.
The condition $\det(P)\neq 0$ in the above corollary usually holds when $A$ is chosen randomly, so we almost can say that $A$ is reflection consistent if and only if $m-r$ is even. We below show that Algorithm \ref{Kaczmarz algorithm} works if $A$ is reflection consistent.
}


\begin{thm}
\label{thm:inconsistent case}
Assume that the linear system $A\x=\b$ is consistent and $A$ is reflection consistent. Then the average of $\{\x_k:k=0,1,2,\ldots\}$ converges to $\x_*$, where $\x_*$ is the solution that has the minimal distance to $\x_0$.
Moreover, 
\be
\label{neweq1}
\|\frac{1}{N} \sum_{i=0}^{N-1} \x_{im} - \x_*\| \leq \varepsilon \|\x_0-\x_*\|,
\ee
where \cs{$N=O(\eta(A)/\varepsilon)$}.
\end{thm}

\begin{proof}
When $A$ is reflection consistent, $\x_0 - \x_*$ does not have an overlap in the subspace generated by the eigenvectors corresponding to the eigenvalue 1. Consequently, the claimed results follow directly from a similar argument to the proof of Theorem \ref{thm2:singular case}.
\end{proof}

As a result, to improve the efficiency of \cs{Algorithm \ref{Kaczmarz algorithm}}, we can reduce the number of linear equations so that $m=r$, where $r$ is the rank. However, we usually do not know $r$ in advance, especially when $m, n$ are large. From Proposition \ref{prop: eigenvalue 1 of products of reflections} and Corollary \ref{cor:a useful result}, a simple way to apply \cs{Algorithm \ref{Kaczmarz algorithm}} is to randomly introduce some new linear constraints so that $m-r$ is even. Even though we do not know $r$ beforehand, we can try introducing an even or an odd number of new linear constraints. With high probability, \cs{Algorithm \ref{Kaczmarz algorithm}} will work in one and exactly one case.
Hence, even if $r$ is not known to us, we can still use \cs{Algorithm \ref{Kaczmarz algorithm}} to find a high-quality solution with high probability. 

As a simple application of the above idea, we can use the number of linear equations to determine the parity of the rank. For instance, let us consider the following two linear systems: 
\be
\label{try}
A\x = \b
\quad {\rm and} \quad
\begin{pmatrix}
A \\
\tilde{\a}^{\TT} \\
\end{pmatrix} \x
=\begin{pmatrix}
\b \\
\tilde{b} \\
\end{pmatrix},
\ee
where $\tilde{\a}=\sum_{i=1}^m \lambda_i A_i$ and $\tilde{b} = \sum_{i=1}^m \lambda_i b_i$ for some arbitrarily chosen $\lambda_1,\ldots,\lambda_m$ so that $\tilde{\a}$ is not the zero vector. When $A$ does not have full row rank, \cs{Algorithm \ref{Kaczmarz algorithm}} only works for one of them.
Consequently, we can use the quality of the solution to check the parity of the rank of $A$. Namely, if we obtain a high-quality solution from the original linear system, then the rank of $A$ has the same parity as the number of linear equations. Otherwise, it has the opposite parity. 

\section{Inconsistent linear systems}
\label{section:Inconsistent linear systems}

In this section, we focus on inconsistent linear systems $A\x=\b$, i.e., the least-squares problem $\arg\min_{\x} \|A \x-\b\|$.
We will show that Algorithm \ref{Kaczmarz algorithm} usually fails to solve the least-squares problem. To this end, we first provide a more clear description of Algorithm \ref{Kaczmarz algorithm}.

Assume that $A$ is $m\times n$ with no zero rows. Let $\x_0$ be an arbitrarily chosen initial vector used in \cs{Algorithm \ref{Kaczmarz algorithm}}, and let $\{\x_0,\x_1,\x_2,\ldots\}$ be the series of vectors generated by \cs{the procedure (\ref{our procedure}).} We introduce an $n\times m$ matrix $F$ depending on $A$ as follows
\be
\label{def:matrix F}
F := \sum_{j=1}^m \frac{2}{\|A_j\|^2} \R_m \cdots \R_{j+1} A_j \e_j^{\TT},
\ee
where $\R_m \cdots \R_{j+1}$ is viewed as the identity matrix when $j=m$. In addition, define 
\be
\label{def:vector u}
\u = F\b.
\ee

\begin{prop}
\label{prop:new identity}
$FA = I_n -\R_A$.
\end{prop}

\begin{proof}
\cs{From equation (\ref{def:matrix F}), we know that
\beas
FA = \sum_{j=1}^m \frac{2}{\|A_j\|^2} \R_m \cdots \R_{j+1} A_j A_j^{\TT} 
= \sum_{j=1}^m  \R_m \cdots \R_{j+1} (I_n - \R_j) 
= I_n -\R_A.
\eeas
This completes the proof.}
\end{proof}

By Proposition \ref{prop:new identity},
\be
\label{relation of two linear systems}
F(A\x - \b) = (I_n - \R_A) \x - \u.
\ee
Below, we show that Algorithm \ref{Kaczmarz algorithm} actually solves the linear system $(I_n - \R_A) \x = \u$ instead of the original one $A\x=\b$.

\begin{thm}
\label{thm:solves new linear system}
Assume that the linear system $(I_n - \R_A) \x = \u$ is consistent. Let $\x_0$ be an arbitrarily chosen initial vector, then $\{\x_0,\x_m,\x_{2m},\x_{3m},\cdots\}$ \cs{generated by the procedure (\ref{our procedure})} lies on the sphere centered on the solution of $(I_n - \R_A) \x = \u$ that has the minimal distance to $\x_0$.
\end{thm}

\begin{proof}
For convenience, we denote the set $\{\x_0,\x_m,\x_{2m},\x_{3m},\cdots\}$ as $\V:=\{\v_0, \v_1, \v_2, \ldots,\}$. Then from the construction of $\x_{jm}$ we know that
\be
\label{shpere eq1}
\v_0 = \x_0, \quad
\v_{k+1} = \R_A \v_k + \u,
\ee
where $\u$ only depends on $A,\b$. It is $\u_m$ determined by the following recursive formula
\[
\u_1 = \frac{2 b_1 A_1}{\|A_1\|^2}, \quad
\u_{k+1} = \R_{k+1} \u_k + \frac{2 b_{k+1} A_{k+1}}{\|A_{k+1}\|^2}~
(k=1,2,\ldots,m-1).
\]
From the above equations, it is not hard to show that $\u = F \b$ defined by (\ref{def:vector u}).

Suppose that $\V$ lies on a sphere, then the center $\c$ must satisfy
\[
\c = 
\lim_{N\rightarrow \infty} \frac{1}{N} \sum_{k=1}^{N} \v_k
=\lim_{N\rightarrow \infty} \frac{1}{N} \sum_{k=1}^{N} (\R_A \v_{k-1} + \u) = \R_A \c + \u .
\]
Hence, $\c$ is the solution of the linear system $(I_n-\R_A) \c = \u$. On the other hand, for this $\c$, by (\ref{shpere eq1}) we can check that $\v_{k+1} - \c = \R_A (\v_k - \c)$. So it is indeed the center of the sphere. 

The proof that $\c$ has the minimal distance to $\x_0$ is similar to that of Theorem \ref{thm1:singular case}. More precisely,
for any solution $\x$ of the linear system $(I_n-\R_A) \x = \u$, we have $\v_{k+1} - \x = \R_A (\v_k - \x)$. So the points in $\V$ have the same distance to any solution of $(I_n-\R_A) \x = \u$. Moreover, the orthogonal property is preserved. Namely, if $\langle \v_0 - \c| \x - \c \rangle = 0$, then $\langle \v_k - \c| \x - \c \rangle = 0$ for all $k$. This means that the center $\c$ has the minimal distance to $\v_0 = \x_0$.
\end{proof}

Next, we explore the conditions such that the two linear systems $A\x = \b$ and $(I_n-\R_A) \x = \u$ are equivalent, i.e., they have the same solution space. The following two propositions further confirm the results we obtained for consistent linear systems. 

\begin{prop}
If ${\rm Rank}(A) = m$, then we have ${\rm Rank}(F) = m$. Consequently, for any $\x$, we have $A\x = \b$ if and only if $(I_n - \R_A)\x = \u$.
\end{prop}

\begin{proof}
By Brady-Watt’s formula, $FA = I_n - \R_A = 2A^{\TT} W^{-1} A$ for some invertible matrix $W$. Since $A$ has full row-rank, we have $F = 2A^{\TT} W^{-1}$. Hence, $F$ has full column-rank.
\end{proof}

More generally, combining Theorem \ref{thm:inconsistent case} we have the following result.

\begin{prop}
\label{cor2:equivalent}
Suppose that $A\x = \b$ is consistent and $A$ is reflection consistent, then for any $\x$, we have $A\x = \b$ if and only if $(I_n - \R_A)\x = \u$.
\end{prop}

\begin{proof}
This is a direct corollary of Theorems \ref{thm:inconsistent case} and \ref{thm:solves new linear system}. To be more exact, from the proof of Theorem \ref{thm:inconsistent case}, we know that Algorithm \ref{Kaczmarz algorithm} will return to a solution of $A\x = \b$ that has the minimal distance to the initial vector. From Theorem \ref{thm:solves new linear system}, we know that this solution is indeed a solution of $(I_n - \R_A)\x = \u$. Any solution of $(I_n - \R_A)\x = \u$ can be obtained in this way by choosing an appropriate initial vector. So the two linear systems $A\x = \b$ and $(I_n-\R_A) \x = \u$ are equivalent.
\end{proof}

The above two results provide more evidence for the success of \cs{Algorithm \ref{Kaczmarz algorithm}} proposed in the above section. 
The following result implies that when $A$ has full column-rank, the linear system $(I_n - \R_A) \x = \u$  has a unique solution.

\begin{prop}
\label{prop:obtain a non singular linear system}
If ${\rm Rank}(A) = n$ and $A$ is reflection consistent, then ${\rm Rank}(I_n - \R_A) = n$.
\end{prop}

\begin{proof}
This follows directly from Theorem \ref{thm:inconsistent case}. More precisely, suppose the first $n$ rows of $A$ are linearly independent. Denote this submatrix as $B$. Since $A$ is reflection consistent,  $\R_A,\R_B$ have the same eigenspace corresponding to \cs{the} eigenvalue $1$. By Brady-Watt’s formula, this eigenspace corresponds to the kernel of $B$ (see the proof of (a) of Proposition \ref{prop: eigenvalue 1 of products of reflections}). But $B$ is invertible, this kernel is empty. Thus $1$ is not an eigenvalue of $\R_A$, i.e., $I_n - \R_A$ is invertible.
\end{proof}

By Equation (\ref{relation of two linear systems}) and Brady-Watt’s formula, to solve $A\x = \b$, \cs{Algorithm \ref{Kaczmarz algorithm}} indeed solves
$A^{\TT} W^{-1} A \x = A^{\TT} W^{-1} \b$, where $W$ is the lower triangular matrix such that $W+W^{\TT}=2AA^{\TT}$. However, when solving a least-squares problem, we usually focus on solving
$A^{\TT} A \x = A^{\TT} \b$. So when the linear system $A\x=\b$ is inconsistent and $A$ does not have full column rank, there is no clear connection between the least-squares solution and the solution of $A^{\TT} W^{-1} A \x = A^{\TT} W^{-1} \b$. In other words, \cs{Algorithm \ref{Kaczmarz algorithm}} fails to solve least-squares problems.

\section{Numerical experiments}
\label{section:Numerical experiments}

In this section, we compare different Kaczmarz algorithms through numerical tests. We focus on four different versions of Kaczmarz algorithms: Strohmer-Vershynin's randomized Kaczmarz algorithm \cite{strohmer2009randomized}, Needell-Tropp's randomized block Kaczmarz algorithm \cite{needell2014paved}, Steinerberger's Kaczmarz algorithm \cite{steinerberger2020surrounding} and a modified version of Algorithm \ref{Kaczmarz algorithm2}. 



For self-containing, we first detail the implementation of these algorithms. Let $A$ be an $m\times n$ matrix, and let $\b$ be an $m\times 1$ vector. For these algorithms, we first arbitrarily choose an initial vector $\x_0$, then iterate over it to obtain a series of vectors. We stop the iteration when we receive a vector $\tilde{\x}$ such that $\|A\tilde{\x}-\b\|\leq \varepsilon$ for some $\varepsilon$.

\begin{itemize}
\item Strohmer-Vershynin's algorithm ({\tt SV}):

Arbitrarily choose an initial guess $\x_0$, and for $k=0,1,2,\ldots$ compute
\be \label{alg1}
\x_{k+1} = \x_k + \frac{b_{i_k} - A_{i_k}^{\TT} \x_k}{\|A_{i_k}\|^2} A_{i_k},
\ee
where $i_k\in\{1,\ldots,m\}$ is chosen with probability $\|A_{i_k}\|^2/\|A\|_F^2$. Output $\x_N$ if $\|A\x_N-\b\|\leq \varepsilon$.


\item {Needell-Tropp's algorithm ({\tt NT}): 

Arbitrarily choose an initial guess $\x_0$, and for $k=0,1,2,\ldots$ compute
\be \label{alg6}
\x_{k+1} = \x_k + A_{\tau}^{+} (\b_{\tau} - A_{\tau}\x_k) ,
\ee
where $A_{\tau}, \b_{\tau}$ are determined as follows.
Let $p=\lceil \|A\|^2 \rceil$, and let $\sigma$ be a permutation of $\{1,\ldots,m\}$.
Define $\mathcal{T} = \{\tau_1,\ldots,\tau_p\}$ as a partition of the row indices $\{1,\ldots,m\}$ with $\tau_i=\{\sigma(k):k=\lceil(i-1)n/m\rceil+1, \cdots, \lceil in/m\rceil\}$.  The matrix $A_{\tau}$ consists of the rows of $A$ whose indices are in $\tau$, and $\b_\tau$ is the subvector of $\b$ with components listed in $\tau$. In each step of the iteration (\ref{alg6}),  $\tau$ is chosen uniformly at random from $\mathcal{T}$, and $A_{\tau}^{+}$ is the pseudoinverse of $A_{\tau}$.
Output $\x_N$ if $\|A\x_N-\b\|\leq \varepsilon$.
}



{

\item Steinerberger's algorithm ({\tt SA}):
\begin{enumerate}
\item Arbitrarily chooses an initial guess $\x_0$.
\item Compute $M \approx \|A\|_F^2\|A^{-1}\|^2$ vectors $\x_0,\ldots,\x_{M-1}$ according to 
\be  \label{alg4}
\x_{k+1} = \x_k + 2\frac{b_{i_k} - A_{i_k}^{\TT} \x_k}{\|A_{i_k}\|^2} A_{i_k},
\ee
where $i_k$ is chosen with probability $\|A_{i_k}\|^2/\|A\|_F^2$.
\item Compute their average
$
\tilde{\x} = \frac{1}{M} \sum_{i=0}^{M-1} \x_i.
$
\item Check if $\|A\tilde{\x} - \b\| \leq \varepsilon$. If yes, return $\tilde{\x}$. Otherwise, go back to step 1 with $\x_0 = \tilde{\x}$.
\end{enumerate}

\item Deterministic iterative reflection algorithm ({\tt DIR}):
\begin{enumerate}
\item Arbitrarily chooses an initial guess $\x_0$.
\item Compute $M \approx m \eta(A)$ vectors $\x_0,\ldots,\x_{M-1}$ according to (\ref{alg4}),
where $i_k = (k \mod m) + 1$ is chosen in order.\footnote{{\color{black}It might be not easy to compute $\eta(A)$. Even if we can compute it very efficiently, if $\eta(A)$ is too large, then it may not be wise to set $M=m\eta(A)$ as the theoretical result suggests. It is time-consuming to iterate many times. So in practice, we can try $M=2^i m$ for a series of $i$ to see which is better. We also do this for Steinerberger’s algorithm. From our tests, these two algorithms perform very well when we choose $i=1-\lfloor \log(m/n)) \rfloor$ or $i=2-\lfloor \log(m/n)) \rfloor$ for the case $m>n$.  }}
\item Compute their average
$
\tilde{\x} = \frac{1}{M} \sum_{i=0}^{M-1} \x_i.
$
\item Check if $\|A\tilde{\x} - \b\| \leq \varepsilon$. If yes, return $\tilde{\x}$. Otherwise, go back to step 1 with $\x_0 = \tilde{\x}$.
\end{enumerate}

}

\end{itemize}

From the analysis in the previous two sections, Algorithm \ref{Kaczmarz algorithm} has some interesting theoretical results; however, this algorithm is usually inefficient in practice because it uses too many rows. Because of this, we mainly focus on {\tt DIR} described above in the numerical tests.
Regarding Steinerberger's algorithm, $A$ is required to be nonsingular. The nonsingularity property ensures the claimed convergence rate. However, this is a theoretical result. In our tests, we will ignore this assumption and run the algorithm {\tt SA} directly.

{Before presenting the numerical results}, we first illustrate the difference between Steinerberger's algorithm and Algorithm \ref{Kaczmarz algorithm2} in dimension 3. 

\begin{figure}[H]
     \centering
     \subfigure[$i_k$ is chosen in order]{
     \includegraphics[height=5cm]{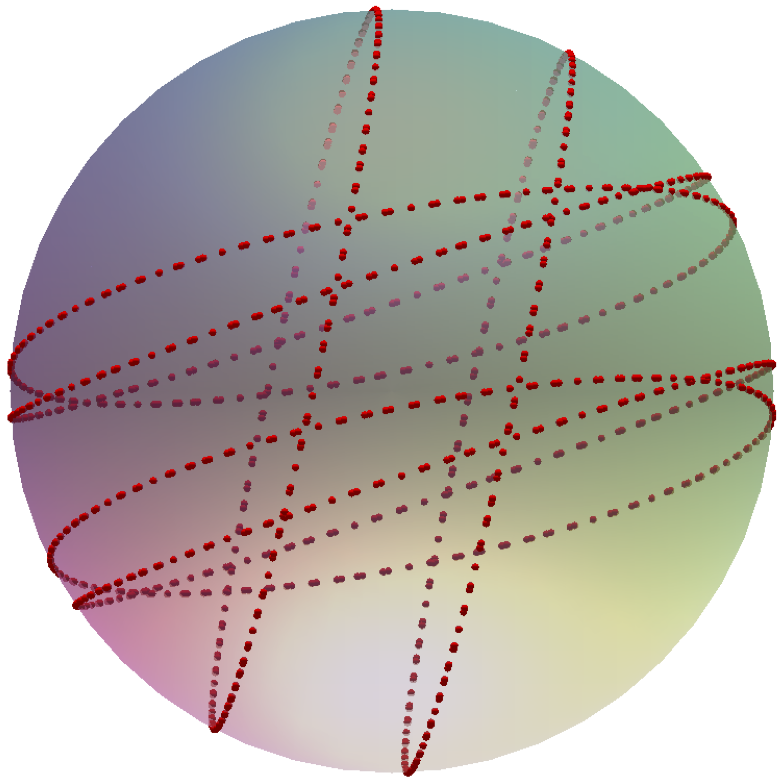}
     }\hspace{3cm}
     \subfigure[$i_k$ is chosen in random]{
     \includegraphics[height=5cm]{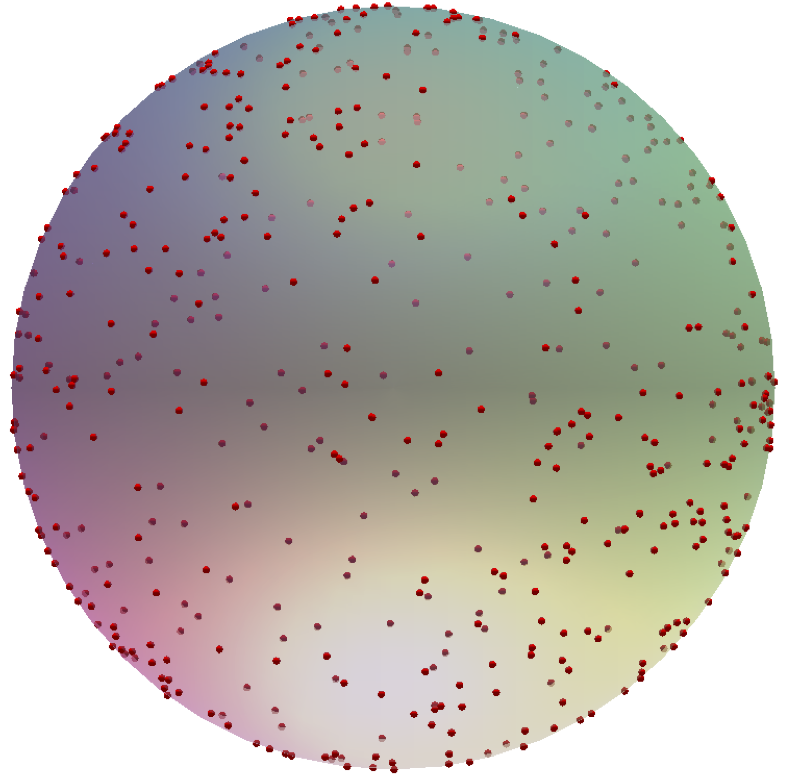}
     }
    \caption{An illustration of the distribution of $\{\x_0, \x_1, \ldots, \x_{1000}\}$ on the sphere when $m=n=3$. The vectors are generated by the procedure (\ref{alg4}).}
    \label{fig1}
\end{figure}

The left one in Figure \ref{fig1}  describes the distribution of the vectors 
$\{\x_i:i=0,1,\ldots,1000\}$ generated by \cs{the} procedure (\ref{alg4}) with $i_k$ chosen in order, i.e., Algorithm \ref{Kaczmarz algorithm2}.
These vectors are located on six separate circles, each circle represents an $\S_{kj}$ defined in Section \ref{section:Consistent linear systems}, and $\S_{k0}$ is parallel to $\S_{k1}$. \cs{The midpoint of the centers of $\S_{k0}$ and $\S_{k1}$ is the solution}. This is confirmed by Theorem \ref{main thm}. 
The right one in Figure \ref{fig1}  shows the distribution of $\{\x_i:i=0,1,\ldots,1000\}$ generated by the process (\ref{alg4}) with $i_k$ chosen with probability $\|A_{i_k}\|^2/\|A\|_F^2$, i.e., Steinerberger's algorithm. The points are randomly distributed on the sphere. The average of them also provides a good approximation of the center of the sphere, i.e., a solution of the linear system $A\x=\b$.

\subsection{Comparison of $\eta(A)$ and $\kappa(A)$}

The complexity of the deterministic iterative reflection algorithm depends on the parameter $\eta(A)$, while the complexities of previous Kaczmarz algorithms depend on the (scaled) condition number of $A$. As discussed after Proposition \ref{prop:relation of eta and kappa}, for random Gaussian matrices, $\eta(A)$ has order $\kappa(A)$ instead of $\kappa^2(A)$. To see this more clearly, we numerically computed $\eta(A)$ and $\kappa(A)$ for around 10000 random matrices of size at most $1000\times 1000$ whose entries are independently taking values from the standard normal distribution. The numerical results suggest that there might be two constants $c_0,c_1$ with $c_0\leq 0.55, c_1\geq 38.27$ such that $c_0 \leq \kappa(A)/\eta(A) \leq c_1$.
\csp{Although it is unclear about the dependence of the ratio $\kappa(A)/\eta(A)$ on the dimension theoretically, the numerical tests here suggest that these two quantities are often comparable.}




\subsection{Comparison of the four algorithms on random matrices}

In this section, we show our testing results. All the calculations were done in the software {\tt Maple 2022} on {\tt MacBook Pro with processor 2.7 GHz Dual-Core Intel Core i5}. The code is available at
\url{https://drive.google.com/file/d/1ZwzSh1Npgiw9nry6IjytkAvtW-wt6Jjh/view?usp=sharing}.
In the tests, we collected the total time the algorithms take to find a solution $\tilde{\x}$ of the linear system $A\x=\b$ such that  $\|A\tilde{\x}-\b\|\leq 0.01$. Since {\tt SA} and {\tt DIR} only apply to consistent linear systems, all the linear systems we generated are consistent. In the tests, we mainly focused on random matrices whose entries independently take values from the standard normal distribution. \csp{In the tests, for each case, we run 50 randomly generated examples and collect the average value of the runtime. All four algorithms start from the same initial vector, which is also generated randomly.}

\setlength{\arrayrulewidth}{0.3mm}
{\renewcommand
\arraystretch{1.5}
\begin{table}[h]
\centering
\begin{tabular}{|c|c|c|c|c|c|c|c|c|c|c|c|} 
 \hline
 $m$ & 200 & 500 & 1000 & 1500 & 2000 & 5000 & 10000 & 15000 & 20000 \\ \hline
 {\tt SV}  & 23.47 & 17.99 & 14.57 & 15.34 & 15.35 & 16.70 & 18.48 & 41.57 & 40.18 \\
 {\tt NT} & 16.83 & 9.73 & 7.29 & 6.53 & 7.46 & 6.79 & 5.75 & 6.43 & 4.45 \\
 {\tt SA} & 2.47 & 1.66 & 2.97 & 1.68 & 2.07 & 2.19 & 2.77 & 3.03 & 3.41 \\
 {\tt DIR} & 2.38 & 2.56 & 2.79 & 2.77 & 3.34 & 3.26 & 4.30 & 4.96 & 5.16 \\ 
 \hline
\end{tabular}
\caption{The runtime (in seconds) of different Kaczmarz methods when $n=100$.}
\label{table:numerical experiments-1}
\end{table}
}

\setlength{\arrayrulewidth}{0.3mm}
{\renewcommand
\arraystretch{1.5}
\begin{table}[h]
\centering
\begin{tabular}{|c|c|c|c|c|c|c|c|c|c|c|c|} 
 \hline
 $m$ & 1000 & 1500 & 2000 & 2500 & 3000 & 3500 & 5000 & 10000 & 15000 & 20000 \\ \hline
 {\tt SV}  & 73.84 & 51.46 & 40.32 & 41.73 & 41.14 & 41.76 & 50.84 & 141.23 & 147.53 & 225.57 \\
 {\tt NT} & 67.59 & 58.98 & 40.46 & 39.51 & 35.85 & 36.21 & 31.95 & 59.09 & 36.17 & 35.21 \\
 {\tt SA} & 6.47 & 6.26 & 5.83 & 6.71 & 8.55 & 9.33 & 10.16 & 18.45 & 24.27 & 26.97 \\
 {\tt DIR} & 7.54 & 7.33 & 7.11 & 8.27 & 11.21 & 12.29 & 13.65 & 22.05 & 25.81 & 29.84 \\ 
 \hline
\end{tabular}
\caption{The runtime (in seconds) of different Kaczmarz methods when $n=300$. }
\label{table:numerical experiments-2}
\end{table}
}

\setlength{\arrayrulewidth}{0.3mm}
{\renewcommand
\arraystretch{1.5}
\begin{table}[h]
\centering
\begin{tabular}{|c|c|c|c|c|c|c|c|c|c|c|c|} 
 \hline
 & $M$ & $2m$ & $m$ & $m/2$ & $m/4$ & $m/8$ & $m/16$ & $\lfloor m/32\rfloor$ & $\lfloor m/64\rfloor$ \\ \hline
$m=10000$ & {\tt SA} & 22.11 & 12.39 & 10.45 & 9.73 & 6.50 & 5.14 & 2.63 & 4.34 \\
$n=100$ & {\tt DIR} & 21.78 & 12.82 & 10.43 & 9.05 & 5.96 & 3.04 & 5.01 & 10.34 \\ \hline
$m=10000$ & {\tt SA} & 120.13 & 58.52 & 41.85 & 26.05 & 17.87 & 14.45 & 14.31 & 23.15 \\
$n=300$ & {\tt DIR} & 115.51 & 57.86 & 36.95 & 26.43 & 22.05 & 34.95 & 75.12 & 654.14 \\  \hline
\end{tabular}
\caption{The influence of the number $M$ on {\tt SA} and {\tt DIR}. The unit of time is seconds.}
\label{table:numerical experiments-3}
\end{table}
}

From Tables \ref{table:numerical experiments-1} and \ref{table:numerical experiments-2}, we can see that {\tt SA} is slightly better than {\tt DIR}, while the difference is not big. Both can compete with {\tt NT}, which is further better than {\tt SV}. Moreover, when the linear system is not so over-determined, {\tt SA} and {\tt DIR} seem to have better performance than {\tt NT}. 
One uncertainty for {\tt SA} and {\tt DIR} is the quantity $M$, which determines how many samples we should generate before we restart further iterations. Let us assume that $m \gg n$. Denote $i=\lfloor \log(m/n) \rfloor$. From our numerical tests, we found that for {\tt SA}, it is more efficient to use $M\in\{\lfloor m/2^{i-1} \rfloor, \lfloor m/2^{i} \rfloor\}$, and for {\tt DIR}, $M\in\{\lfloor m/2^{i-2} \rfloor, \lfloor m/2^{i-1} \rfloor\}$ is better. In Table \ref{table:numerical experiments-3}, we list some numerical results for this. We also found that for {\tt DIR}, it is very slow if we set $M\leq n$ (e.g., see the last example in the fifth row of Table \ref{table:numerical experiments-3}). This is not a problem for {\tt SA}. But the efficiency of {\tt SA} may be affected if $M\ll n$ (e.g., see the last example in the fourth row of Table \ref{table:numerical experiments-3}). Both are inefficient when $M$ is too large.

\section{Conclusions}

In this paper, we proposed a new deterministic Kaczmarz algorithm by replacing orthogonal projections with reflections in the original Kaczmarz method. We established rigorous theorems about its correctness and efficiency. In this process, we discovered an interesting geometric fact about the solutions of consistent linear systems. Namely, any solution of a consistent linear system can be represented as the midpoint of the centers of two spheres. We feel it would be interesting to find more results from this fact. According to our numerical tests, the new Kaczmarz algorithm has a performance competitive with randomized Kaczmarz algorithms for consistent linear systems. However, it cannot be used to solve least-squares problems, which is partially caused by the structure of the algorithm. It would be interesting to find a way to fix this. Perhaps some new interesting geometric facts can be found in this way. Another interesting research topic is to find certain matrices such that $\eta(A)$ is small. In this paper, we found some partial results about its connection to the condition number. But this might be worth a systematic study. Finally, from our numerical experiments,  Steinerberger's algorithm works very well in practice. However, Steinerberger only studied the convergence rate in the nonsingular case. So it would be interesting to better understand this algorithm theoretically in the general case.

\section*{Acknowledgement}

I would like to thank Alex Little and Nina Snaith for the helpful discussions on Proposition \ref{prop:relation of eta and kappa}, and the anonymous referees for valuable suggestions which greatly improved this work. I also would like to thank Gilbert Strang for his helpful suggestions on the paper.
I acknowledge support from EPSRC grant EP/T001062/1. This project has received funding from the European Research Council (ERC) under the European Union's Horizon 2020 research and innovation programme (grant agreement No.\ 817581).

\bibliographystyle{siam}
\bibliography{Amain}

\begin{thebibliography}{10}

\bibitem{agmon1954relaxation}
{\sc S.~Agmon}, {\em The relaxation method for linear inequalities}, Canadian
  Journal of Mathematics, 6 (1954), pp.~382--392.

\bibitem{angelos1992triangular}
{\sc J.~R. Angelos, C.~C. Cowen, and S.~K. Narayan}, {\em {Triangular
  truncation and finding the norm of a Hadamard multiplier}}, Linear Algebra
  and its Applications, 170 (1992), pp.~117--135.

\bibitem{ansorge1984connections}
{\sc R.~Ansorge}, {\em {Connections between the Cimmino-method and the
  Kaczmarz-method for the solution of singular and regular systems of
  equations}}, Computing, 33 (1984), pp.~367--375.

\bibitem{benzi2004gianfranco}
{\sc M.~Benzi}, {\em {Gianfranco Cimmino’s contributions to numerical
  mathematics}}, Atti del Seminario di Analisi Matematica, Dipartimento di
  Matematica dell’Universita di Bologna. Volume Speciale: Ciclo di Conferenze
  in Ricordo di Gianfranco Cimmino,  (2004), pp.~87--109.

\bibitem{brady2006products}
{\sc T.~Brady and C.~Watt}, {\em {On products of Euclidean reflections}}, The
  American Mathematical Monthly, 113 (2006), pp.~826--829.

\bibitem{censor1981row}
{\sc Y.~Censor}, {\em Row-action methods for huge and sparse systems and their
  applications}, SIAM review, 23 (1981), pp.~444--466.

\bibitem{censor1983strong}
{\sc Y.~Censor, P.~P. Eggermont, and D.~Gordon}, {\em {Strong underrelaxation
  in Kaczmarz's method for inconsistent systems}}, Numerische Mathematik, 41
  (1983), pp.~83--92.

\bibitem{chen2021fast}
{\sc J.-Q. Chen and Z.-D. Huang}, {\em {On a fast deterministic block Kaczmarz
  method for solving large-scale linear systems}}, Numerical Algorithms,
  (2021), pp.~1--23.

\bibitem{Cimmino}
{\sc G.~Cimmino}, {\em Calcolo approssimato per le soluzioni dei sistemi di
  equazioni lineari}, La Ricerca Scientifica, II (1938), pp.~326--333.

\bibitem{matrices1986random}
{\sc J.~E. Cohen, H.~Kesten, and C.~M. Newman}, {\em {Random Matrices and Their
  Applications: Proceedings of the AMS-IMS-SIAM Joint Summer Research
  Conference Held June 17-23, 1984, with Support from the National Science
  Foundation}}, vol.~50 of Contemporary mathematics, American Mathematical
  Society, Providence, Rhode Island, 1986.

\bibitem{Deutsch1985}
{\sc F.~Deutsch}, {\em Rate of Convergence of the Method of Alternating
  Projections}, Birkh{\"a}user Basel, Basel, 1985, pp.~96--107.

\bibitem{DEUTSCH1997381}
{\sc F.~Deutsch and H.~Hundal}, {\em {The Rate of Convergence for the Method of
  Alternating Projections, II}}, Journal of Mathematical Analysis and
  Applications, 205 (1997), pp.~381--405.

\bibitem{edelman1988eigenvalues}
{\sc A.~Edelman}, {\em Eigenvalues and condition numbers of random matrices},
  SIAM journal on matrix analysis and applications, 9 (1988), pp.~543--560.

\bibitem{eggermont1981iterative}
{\sc P.~P.~B. Eggermont, G.~T. Herman, and A.~Lent}, {\em Iterative algorithms
  for large partitioned linear systems, with applications to image
  reconstruction}, Linear Algebra and its Applications, 40 (1981), pp.~37--67.

\bibitem{eldar2011acceleration}
{\sc Y.~C. Eldar and D.~Needell}, {\em {Acceleration of randomized Kaczmarz
  method via the Johnson--Lindenstrauss lemma}}, Numerical Algorithms, 58
  (2011), pp.~163--177.

\bibitem{feichtinger1992new}
{\sc H.~G. Feichtinger, C.~Cenker, M.~Mayer, H.~Steier, and T.~Strohmer}, {\em
  New variants of the pocs method using affine subspaces of finite codimension
  with applications to irregular sampling}, in Visual Communications and Image
  Processing'92, vol.~1818, SPIE, 1992, pp.~299--310.

\bibitem{GALANTAI200530}
{\sc A.~Gal\'{a}ntai}, {\em On the rate of convergence of the alternating
  projection method in finite dimensional spaces}, Journal of Mathematical
  Analysis and Applications, 310 (2005), pp.~30--44.

\bibitem{GoluVanl96}
{\sc G.~H. Golub and C.~F. Van~Loan}, {\em Matrix Computations}, The Johns
  Hopkins University Press, fourth~ed., 2013.

\bibitem{GORDON1970471}
{\sc R.~Gordon, R.~Bender, and G.~T. Herman}, {\em {Algebraic Reconstruction
  Techniques (ART) for three-dimensional electron microscopy and X-ray
  photography}}, Journal of Theoretical Biology, 29 (1970), pp.~471--481.

\bibitem{gower2019adaptive}
{\sc R.~Gower, D.~Molitor, J.~Moorman, and D.~Needell}, {\em Adaptive
  sketch-and-project methods for solving linear systems}, arXiv preprint
  arXiv:1909.03604,  (2019).

\bibitem{gower2015randomized}
{\sc R.~M. Gower and P.~Richt{\'a}rik}, {\em Randomized iterative methods for
  linear systems}, SIAM Journal on Matrix Analysis and Applications, 36 (2015),
  pp.~1660--1690.

\bibitem{hildreth1957quadratic}
{\sc C.~Hildreth et~al.}, {\em A quadratic programming procedure}, Naval
  research logistics quarterly, 4 (1957), pp.~79--85.

\bibitem{jarman2021randomized}
{\sc B.~Jarman, N.~Mankovich, and J.~D. Moorman}, {\em Randomized extended
  kaczmarz is a limit point of sketch-and-project}, arXiv preprint
  arXiv:2110.05605,  (2021).

\bibitem{jiao2017preasymptotic}
{\sc Y.~Jiao, B.~Jin, and X.~Lu}, {\em {Preasymptotic convergence of randomized
  Kaczmarz method}}, Inverse Problems, 33 (2017), p.~125012.

\bibitem{karczmarz1937angenaherte}
{\sc S.~Karczmarz}, {\em Angenaherte auflosung von systemen linearer
  glei-chungen}, Bull. Int. Acad. Pol. Sic. Let., Cl. Sci. Math. Nat.,  (1937),
  pp.~355--357.

\bibitem{ma2015convergence}
{\sc A.~Ma, D.~Needell, and A.~Ramdas}, {\em {Convergence properties of the
  randomized extended Gauss--Seidel and Kaczmarz methods}}, SIAM Journal on
  Matrix Analysis and Applications, 36 (2015), pp.~1590--1604.

\bibitem{meckes2019random}
{\sc E.~S. Meckes}, {\em The random matrix theory of the classical compact
  groups}, vol.~218, Cambridge University Press, 2019.

\bibitem{moorman2021randomized}
{\sc J.~D. Moorman, T.~K. Tu, D.~Molitor, and D.~Needell}, {\em {Randomized
  Kaczmarz with averaging}}, BIT Numerical Mathematics, 61 (2021),
  pp.~337--359.

\bibitem{motzkin1954relaxation}
{\sc T.~S. Motzkin and I.~J. Schoenberg}, {\em The relaxation method for linear
  inequalities}, Canadian Journal of Mathematics, 6 (1954), pp.~393--404.

\bibitem{necoara2019faster}
{\sc I.~Necoara}, {\em {Faster randomized block Kaczmarz algorithms}}, SIAM
  Journal on Matrix Analysis and Applications, 40 (2019), pp.~1425--1452.

\bibitem{needell2014paved}
{\sc D.~Needell and J.~A. Tropp}, {\em {Paved with good intentions: analysis of
  a randomized block Kaczmarz method}}, Linear Algebra and its Applications,
  441 (2014), pp.~199--221.

\bibitem{needell2014stochastic}
{\sc D.~Needell, R.~Ward, and N.~Srebro}, {\em {Stochastic gradient descent,
  weighted sampling, and the randomized Kaczmarz algorithm}}, Advances in
  neural information processing systems, 27 (2014), pp.~1017--1025.

\bibitem{needell2015randomized}
{\sc D.~Needell, R.~Zhao, and A.~Zouzias}, {\em {Randomized block Kaczmarz
  method with projection for solving least squares}}, Linear Algebra and its
  Applications, 484 (2015), pp.~322--343.

\bibitem{niu2020greedy}
{\sc Y.-Q. Niu and B.~Zheng}, {\em {A greedy block Kaczmarz algorithm for
  solving large-scale linear systems}}, Applied Mathematics Letters, 104
  (2020), p.~106294.

\bibitem{nutini2016convergence}
{\sc J.~Nutini, B.~Sepehry, A.~Virani, I.~Laradji, M.~Schmidt, and H.~Koepke},
  {\em Convergence rates for greedy kaczmarz algorithms}, in Conference on
  Uncertainty in Artificial Intelligence, 2016.

\bibitem{petra2016single}
{\sc S.~Petra and C.~Popa}, {\em {Single projection Kaczmarz extended
  algorithms}}, Numerical Algorithms, 73 (2016), pp.~791--806.

\bibitem{richtarik2020stochastic}
{\sc P.~Richt{\'a}rik and M.~Tak{\'a}c}, {\em Stochastic reformulations of
  linear systems: algorithms and convergence theory}, SIAM Journal on Matrix
  Analysis and Applications, 41 (2020), pp.~487--524.

\bibitem{shao-montanaro21}
{\sc C.~Shao and A.~Montanaro}, {\em Faster quantum-inspired algorithms for
  solving linear systems}, ACM Transactions on Quantum Computing, 3 (2022),
  pp.~1--23.

\bibitem{shao2020row}
{\sc C.~Shao and H.~Xiang}, {\em Row and column iteration methods to solve
  linear systems on a quantum computer}, Physical Review A, 101 (2020),
  p.~022322.

\bibitem{steinerberger2020surrounding}
{\sc S.~Steinerberger}, {\em Surrounding the solution of a linear system of
  equations from all sides}, Quarterly of Applied Mathematics, 79 (2021),
  pp.~419--429.

\bibitem{strohmer2009randomized}
{\sc T.~Strohmer and R.~Vershynin}, {\em {A randomized Kaczmarz algorithm with
  exponential convergence}}, Journal of Fourier Analysis and Applications, 15
  (2009), pp.~262--278.

\bibitem{tanabe1971projection}
{\sc K.~Tanabe}, {\em Projection method for solving a singular system of linear
  equations and its applications}, Numerische Mathematik, 17 (1971),
  pp.~203--214.

\bibitem{tao2010smooth}
{\sc T.~Tao and V.~Vu}, {\em Smooth analysis of the condition number and the
  least singular value}, Mathematics of computation, 79 (2010), pp.~2333--2352.

\bibitem{wu2020projected}
{\sc N.~Wu and H.~Xiang}, {\em {Projected randomized Kaczmarz methods}},
  Journal of Computational and Applied Mathematics, 372 (2020), p.~112672.

\bibitem{xiang2017randomized}
{\sc H.~Xiang and L.~Zhang}, {\em Randomized iterative methods with alternating
  projections}, arXiv preprint arXiv:1708.09845,  (2017).

\bibitem{yaniv2021selectable}
{\sc Y.~Yaniv, J.~D. Moorman, W.~Swartworth, T.~Tu, D.~Landis, and D.~Needell},
  {\em {Selectable Set Randomized Kaczmarz}}, arXiv preprint arXiv:2110.04703,
  (2021).

\bibitem{zouzias2013randomized}
{\sc A.~Zouzias and N.~M. Freris}, {\em {Randomized extended Kaczmarz for
  solving least squares}}, SIAM Journal on Matrix Analysis and Applications, 34
  (2013), pp.~773--793.

\end{thebibliography}

\end{document}